\documentclass[a4paper,12pt]{article}

\usepackage{amsmath,amsthm,amssymb,sansmath,color,geometry,multicol,graphicx, float,color,authblk, hyperref,stmaryrd,enumerate}
\usepackage[english]{babel}
\usepackage[utf8]{inputenc}
\newtheorem{thm}{Theorem}
\newtheorem{prop}[thm]{Proposition}

\newtheorem{lem}[thm]{Lemma}
\newtheorem{cor}[thm]{Corollary}

\newtheorem*{defi*}{Definition}
\newtheorem*{hyp*}{Assumptions}

\newcommand{\x}{\partial_x}
\newcommand{\y}{\partial_y}

\newcommand{\po}{\left(}
\newcommand{\pf}{\right)}
\newcommand{\co}{\left[}
\newcommand{\cf}{\right]}
\newcommand{\cco}{\llbracket}
\newcommand{\ccf}{\rrbracket}
\newcommand{\R}{\mathbb R}

\newcommand{\dd}{\text{d}}
\newcommand{\A}{\mathcal A}

\newcommand{\na}{\nabla}
\newcommand{\X}{\textbf{X}}
\newcommand{\tX}{\widetilde{\textbf{X}}}
\newcommand{\hX}{\widehat{\textbf{X}}}

\title{Generalized $\Gamma$ calculus and application to interacting particles on a graph.}
\author{Pierre Monmarché\footnote{Laboratoire Jacques-Louis Lions, Sorbonne Universit\'e, 4 place Jussieu 75005 Paris, France. Email:
 			pierre.monmarche@upmc.fr}}
\date{Februar 22, 2018\footnote{First version October 20, 2015}}

\begin{document}

\maketitle

\begin{abstract}
The classical Bakry-\'Emery calculus is extended to study, for degenerated (non-elliptic, non-reversible, or non-diffusive) Markov processes, questions such as hypoellipticity, hypocoercivity, functional  inequalities or Wasserstein contraction. In particular we obtain the optimal speed of convergence to equilibrium for any ergodic Ornstein-Uhlenbeck process, which is given by the spectral gap of the drift matrix and the size of the corresponding Jordan blocks. We also study chains of $N$ interacting overdamped particles and establish for their invariant measures log-Sobolev inequalities with constants of order $N^2$, which is optimal.
\end{abstract}

\section{Introduction and overview}

The $\Gamma$ calculus has been introduced by Bakry and \'Emery in \cite{BakryEmery}. When $L$ is a Markov generator, define the carr\'e du champ operator by
\[\Gamma (f,g) = \frac12 \po L(fg)  - g L f - fL g\pf\]
and the $\Gamma_2$ operator by
\[\Gamma_2 (f,g) = \frac12 \po L \Gamma (f,g) - \Gamma(f,Lg) - \Gamma(g,Lf)\pf.\] 
These operators are related to various properties of $L$, of its associated semigroup $(P_t)_{t\geq 0} = \po e^{tL} \pf_{t\geq 0}$ and, as the case may be, its invariant measure $\mu$: hypercontractivity, regularity, long-time behaviour, concentration of measure, etc. The main reference on this topic is certainly the book \cite{BakryGentilLedoux} of Bakry, Gentil and Ledoux; nevertheless in the rest of this work we will mostly refer to \cite{BolleyGentil}, which allow for a shorter introduction. This classical approach is for instance very efficient to study the reversible elliptic diffusion with generator
\[L = - \na U.\na + \Delta\]
where $\Delta$ is the Laplace-Beltrami operator of a Riemannian manifold $\mathcal M$ and $U$ is a smooth potential on $\mathcal M$; especially when $U$ is strongly convex.

While this theory is still a vivid area of research, in parallel of its developments, interest has grown for some decades in so-called "degenerated" processes for which the classical methods do not work, at least not in their usual form. The typical example is the kinetic diffusion on $\R^{2d}$ with generator
\[L f(x,y) = \co y.\na_x + (\na_x U(x) + y). \na_y + \Delta_y \cf f(x,y), \]
which has inspired an extensive amount of works, among which we will only cite for now the memoir \cite{Villani2009}. Piecewise deterministic Markov processes form another large class of such degenerated processes (see \cite{MalrieuPDMP}).
 
 The highlights of Villani's memoir include the idea that when classical quantities (variance, entropy, energy\dots) do not behave well under the action of the semigroup, it may be helpful to work with ad hoc distorted quantities for which usual techniques may apply. This heuristic have proved useful in a variety of situations, in various ways (see for instance \cite{AntonErb,DMS2009,MonmarchePDMP}). The drawback is a (possibly drastic) complexification of the computations.

\bigskip

The aim of this paper is to give a somewhat unified framework in which the computations appear relatively nice. This is done in Section \ref{SectionGamma}. The strategy closely follows the Bakry-\'Emery theory (and therefore we call it "generalised $\Gamma$ calculcus"), except we consider a broader variety of functionals. This idea is already present in \cite{Baudoin,AntonErb,MonmarchePDMP}.

\bigskip

For the sake of clarity we will try to avoid technicalities, and in particular we won't address the question of the optimal functional spaces for which our results hold. All the test functions will be considered in a space $\A$ for which very strong assumptions will hold. In applications, it may often be taken as the set of $\mathcal{C}^\infty$ functions which are bounded below by a positive constant and whose derivatives all grow at most polynomially at infinity.

In order to give an overview of possible methods, examples of applications will systematically be provided. Most of them come from existing works. 
As we will see the $\Gamma$ calculus point of view sometimes improves and almost always simplifies them. For instance, consider the case of Ornstein-Uhlenbeck processes, with generator
\[L f(x) = -(Bx).\na f(x) + \text{div}\po D \na f\pf(x),\]
where $B$ and $D$ are constant matrices. Suppose ker$D$ does not contain any non trivial subspace which is invariant by $B^T$, and $\rho = \inf\{\Re (\nu), \ \nu \text{ eigenvalue of }B\} >0$. Then the associated semigroup $(P_t)_{t\geq 0}$ admits a unique invariant measure $\mu$.  Denote by
\[\text{Ent}_\mu f = \int f \ln f \dd \mu - \int f\dd \mu \ln \int f \dd \mu\] 
the relative entropy of a function $f$ with respect to $\mu$. Arnold and his co-authors (\cite{AntonErb,Arnold2015}) proved (working with $L'$ the dual of $L$ with respect to the Lebesgue measure, but this boils down to the same) that for all $\varepsilon>0$ there exists a constant $c_\varepsilon>0$ such that for all $f$ with $\text{Ent}_\mu f < \infty$,
\[\text{Ent}_\mu( P_t f) \leq c_\varepsilon e^{-2(\rho-\varepsilon) t} \text{Ent}_\mu f,\]
where $\varepsilon$ may be taken equal to $0$ with a finite $c_0$ if $B$ is diagonalisable and more generally if none of the eigenvalues in $\{\Re(\lambda) = \rho\}$ is defective. We will prove (Corollary \ref{CorOU} below) that in fact, when $N$ is the maximal size of a Jordan block associated to an eigenvalue in $\{\Re(\lambda) = \rho\}$ in the decomposition of $B$, then there exists a constant $c$ such that 
\[\text{Ent}_\mu( P_t f) \leq c (1 + t^{2(N-1)}) e^{-2\rho t} \text{Ent}_\mu f.\]
The power $t^{2(N-1)}$ is optimal (see Proposition \ref{PropOU}).

\bigskip

Apart from these revisited versions of known results, our main interest lies in the study of systems of attracting particles on a graph. This is done in Section \ref{SectionGraph}. Let $W$ be an even, smooth, strongly convex potential on $\R^d$ and let $\textbf{X}=(X_i)_{i=0..N}$ be a chain of $N+1$ particles interacting with their neighbours via the dynamics
\[\left\{\begin{array}{rcl}
\dd X_0 & =& -\na W(X_0-X_{1}) \dd t + \sqrt T_0 \dd  B_t\\
\dd X_i  & = & - \po \na W(X_i - X_{i-1}) + \na W(X_i - X_{i+1}) \pf \dd t\hspace{20pt}\text{for }i=1..(N-1)\\
\dd X_N & =& -\na W(X_N-X_{N-1}) \dd t + \sqrt T_N \dd \widetilde  B_t
\end{array}\right. \]
where $B$ and $\widetilde B$ are independant standard Brownian motion on $\R^d$ and $T_0,T_N>0$.

The process $\textbf{X}$ is a toy model related to the chains of Hamiltonian oscillators introduced by Eckmann, Pillet  and Rey-Bellet in \cite{EckPRB1999} (see also the more recent \cite{HairerHeatBath} and references within). Except in some particular cases, the invariant measures of such processes are not explicit, and therefore it is not clear, for instance, whether they satisfy some functional inequality such as the Poincar\'e or log-Sobolev ones (see the discussion in \cite[$\S$ 9.2 p.67]{Villani2009}; also note that the results of \cite{Cattiaux2008,CattiauxGuillinPAZ} do not apply since the dynamics is not reversible). 

If $\overline X = \frac1{N+1} \sum_{i=0}^N X_i$ is the center of mass of the system, then $\textbf{X}-\overline X$ is ergodic. The process $\textbf{X}$ is highly degenerated, since the randomness of a system in $\R^{dN}$ is only given by a Brownian motion on $\R^{2d}$. Nevertheless we will prove for $\textbf{X} -\overline X$ (and for similar models) the invariant measure $\mu_N$ satisfies a log-Sobolev inequality
\begin{eqnarray*}
\text{Ent}_{\mu_N} f &  \leq  & c N^2 \int \frac{|\na f|^2}{f} \dd \mu_N\hspace{30pt}\forall f\in\A \text{ s.t. } \sum_{i=0}^N \na_ i f = 0,
\end{eqnarray*}
where $c$ does not depend on $N$ (the condition $\sum \na_i f=0$ compensate the fact the center of mass is fixed). The $N^2$ is optimal since, when $W(x) = \frac12 |x|^2$,  it can be checked the optimal constant in the log-Sobolev inequality is indeed of order $N^2$.

\subsubsection*{General notations and conventions}

The following holds throughout the whole paper. A vector $x\in\R^d$ will always be equated to a 1-column matrix, and therefore the scalar product in $\R^d$ will indifferently be written $x.y$ or $x^Ty$ where $A^T$ denotes the transpose of a matrix $A$. Unless otherwise specified, $|x|$ is the Euclidian norm of $x$, $|A|$ is the operator norm $|A| = \sup \{|A x|,\ |x|=1\}$ and $\sigma(A)$ is the spectrum of $A$. When $Q$ is a quadratic operator, we still call $f,g \mapsto Q(f,g) = \frac14\po Q(f+g)-Q(f-g)\pf$ the associated bilinear symmetric operator obtained by polarization. When $A$ and $B$ are symmetric matrices and $\rho \in \R$ then $A\geq B$ (resp. $A\geq \rho$) means $x^T A x \geq x^T B x$ (resp. $\geq \rho |x|^2$) for all $x\in\R^d$.  We note $\mathcal P(E)$ the set of probability laws on a space $E$ and for $\mu\in\mathcal P(E)$ in some cases we use the operator notation $\mu f = \int f\dd \mu$. When $x\mapsto b(x)$ is a smooth vector field on $\R^d$ we note $J_b(x) = (\partial_i b_j(x))_{1\leq i,j\leq d}$  the Jacobian matrix of $b$ (the $i$ is for the line, the $j$ for the column). In particular, if $b(x) = A x$ with a constant matrix $A$, then $J_b(x) = A^T$ for all $x$.

\section{Gamma Calculus}\label{SectionGamma}

We recall here some ideas from \cite{MonmarcheRecuitHypo}. For a bounded measurable $f : E\rightarrow \R$, let $P_t f(x) = \mathbb E\po f(X_t)\ | \ X_0=x \pf$ be a Markov semi-group associated to a Markov process $X$ on a Polish space $E$, $Lf(x) = \lim_{t\rightarrow0} \frac1 t(P_t f(x) - f(x))$ be its generator for $f$ in $\mathcal D(L)$ the domain of $L$. Assume $\mathcal D(L)$ contains a core $\A$ which is an algebra fixed by $L$ and $P_t$. Let $\A_+=\{ f\in\A, \ f \geq 0\}$ and let $\Phi : \A_+ \rightarrow \A$ be differentiable with respect to pointwise convergence topology with differential operator $\text{d} \Phi$, in the sense $\lim_{s\rightarrow0}\frac1s\po \Phi(f+s g) - \Phi(f)\pf(x) = \po \text{d} \Phi(f).g\pf(x)$ for all $x\in E, f,g\in\A_+$. For $f\in\A_+$ we define
\[\Gamma_{L,\Phi}(f) = \frac12 \po L \Phi(f) - \text d\Phi(f).Lf\pf\]
and we suppose $\Gamma_{L,\Phi}(f) \in \A$. When there is no ambiguity on the generator we only write $\Gamma_\Phi$, and when $\Phi(f) = f^2+af+b$ with $(a,b)\in\R^2$ we retrieve the usual carr\'e du champ operator which is simply denoted by $ \Gamma$. 
When $\Phi(f) = \Gamma(f)$, we retrieve the classical Bakry-\'Emery $\Gamma_2$ operator, since for a quadratic form $Q(f)$, $\text d Q(f).g = 2 Q(f,g)$.

\bigskip

\textbf{Remark:} The factor $\frac12$ in the definition of $\Gamma_{L,\Phi}$ is disputable, since with this definition we will most of the time deal with $2\Gamma_{L,\Phi}$, but we wanted to be consistent with the classical definitions of $\Gamma$ and $\Gamma_2$. The $\frac12$ in the carr\'e du champ ensures that $\Gamma f = |\na f|^2 $ when $L=\Delta$ is the Laplace-Beltrami of a Riemannian manifold. Having said that, the generator of a standard Brownian motion is $\frac12 \Delta$, which advocates for a definition with no $\frac12$.

\subsection{Motivations}

Without claiming to be comprehensive, we want to emphasize in this section some reasons why $\Gamma_\Phi$ operators may be useful and especially why it may be interesting to work with more general functions than only $\Gamma$ (or functions based on $\Gamma$, such as in \cite{BolleyGentil} and references within). The examples below come from other works to which we will refer for complete proofs and computations.

\subsubsection{Sub-commutation $\Phi$/semigroup}

The $\Gamma_\Phi$ operator naturally appears in the interpolation from $\Phi(P_t f)$ to $P_t \Phi(f)$. Indeed, for $f\in\A_+$, $t>0$, $x\in E$ and $s\in[0,t]$, let $\psi(s) = P_s \Phi( P_{t-s} f)(x)$. Then
\begin{eqnarray*}
\psi'(s) & = & 2 P_s \Gamma_\Phi(P_{t-s} f) (x).
\end{eqnarray*}
As a direct consequence, we have the following:
\begin{lem}\label{LemGammaCurvature}
The so-called curvature condition $\Gamma_{\Phi} \geq \rho \Phi$ for some $\rho \in \R$, is equivalent to 
\begin{eqnarray}\label{EqIntertwinning}
\forall f\in \A_+\hspace{15pt}\Phi\po P_t f \pf & \leq & e^{-2\rho t}P_t \Phi\po f\pf.
\end{eqnarray}
\end{lem}
\begin{proof}
If the curvature condition holds, the Gronwall's Lemma yields $\psi(0) \leq e^{-2\rho t}\psi(t)$, which is \eqref{EqIntertwinning}. Differentiating \eqref{EqIntertwinning} at $t=0$ yields the curvature condition.
\end{proof}

\textbf{Example 1:} Consider the TCP process on $\R_+$ with generator
\begin{eqnarray*}
Lf(x) & = & f'(x) + \lambda \po \int_{h=0}^\infty f\po h x\pf\dd \nu( h) - f(x)\pf
\end{eqnarray*}
with $\lambda>0$, an $\nu\in\mathcal P(\R_+)$. For instance with $\nu = \delta_{h_0}$, we compute $\Gamma f(x) = \frac\lambda2 \po f(h_0 x) - f(x)\pf^2$ and (possibly with the help of Lemma~\ref{LemGammaQuadratic} below)
\[\Gamma_2 f(x) = \frac{\lambda^2}4\po \po f(h_0^2 x) - 2 f(h_0 x) + f(x) \pf^2 + 2(h_0-1)f'(h_0 x)\po f(h_0 x)-f(x)\pf\pf,\]
so that no curvature condition $\Gamma_2 \geq \rho \Gamma$ can hold. In contrast, with $\Phi(f) = |\na f|^2$, it is proved in \cite[Proposition 15]{MonmarchePDMP} that $\Gamma_\Phi \geq \frac{\lambda}2 \po 1 - \int h^2 \dd \nu(h)\pf\Phi$ and thus
\begin{eqnarray*}
|\na P_t f|^2 & \leq & e^{-\lambda \po 1 - \int h^2 \dd \nu(h)\pf t} P_t|\na f|^2.
\end{eqnarray*}

\subsubsection{Functional inequalities}
 
In the general case, if $\mu$ is an invariant law of $P_t$, integrating \eqref{EqIntertwinning} reads
\begin{eqnarray*}
\int \Phi(P_t f) \dd \mu & \leq & e^{-2\rho t} \int \Phi(f) \dd \mu.
\end{eqnarray*}
We say $P_t$ is ergodic if it admits a unique invariant law $\mu$ such that for all $f\in \A,x\in E$, $P_t f(x) \rightarrow \int f \dd \mu$ as $t\rightarrow \infty$. In this case we suppose $\A \subset L^2(\mu)$.

\begin{lem}\label{LemInegalitePoincare}
Suppose $P_t$ is ergodic and $\Gamma_{\Phi_1}(P_t f) \leq \gamma(t) P_t \Phi_2(f)$ for all $f\in\A_+$, with $\gamma \in L^1(\R_+)$. Then
\begin{eqnarray}\label{EqInegalIntegre}
\forall f\in\A_+,\hspace{20pt}\int \Phi_1(f) \dd \mu - \Phi_1\po \int f \dd \mu\pf & \leq & 2\po\int_0^\infty \gamma(s) \dd s \pf \int \Phi_2(f) \dd \mu.
\end{eqnarray}
In particular if the "generalized $\Gamma_2$ condition" $\Gamma_{\Gamma_{\Phi_1}} \geq \rho \Gamma_{\Phi_1}$ holds with $\rho>0$, then
\begin{eqnarray*}
\forall f\in\A_+,\hspace{20pt}\int \Phi_1(f) \dd \mu - \Phi_1\po \int f \dd \mu\pf & \leq & \frac1\rho \int \Gamma_{\Phi_1}(f) \dd \mu.
\end{eqnarray*} 
\end{lem}
\begin{proof}
If $\psi_1(s) = P_s \Phi_1(P_{t-s} f)$,
\begin{eqnarray}\label{EqInegalPonctuelle}
P_t \Phi_1(f) - \Phi_1(P_t f) & = & \psi_1(t) - \psi_1(0) \notag\\
&= & 2 \int_0^t P_s \Gamma_{\Phi_1}(P_{t-s} f) \dd s \notag\\
&  \leq &  2\po \int_0^t \gamma(s)\dd s\pf P_t \Phi_2(f).
\end{eqnarray}
Conclusion follows when $t\rightarrow\infty$. 
When the generalized $\Gamma_2$ condition holds, according to Lemma \ref{LemGammaCurvature} we may take $\Phi_2 = \Gamma_{\Phi_1}$ and $\gamma(t) = e^{-2\rho t}$.
\end{proof}
\textbf{Remark:} An inequality of the form \eqref{EqInegalPonctuelle} is called a local functional inequality: it is a functional inequality satisfied by all the measures $\po P_t(x) \pf_{x\in E} = \po \delta_x P_t \pf_{x\in E}$, uniformly in $x\in E$. By contrast an inequality of the form \eqref{EqInegalIntegre} may be called an integrated functional inequality. A local inequality may hold even if $\gamma \notin L^1(\R_+)$: for instance for the Brownian motion in $\R^d$, $L = \frac12 \Delta$ satisfies the curvature condition $\Gamma_2 \geq 0$, and thus $P_t f^2 - (P_tf)^2 \leq t P_t |\na f|^2$.

\bigskip

\textbf{Example 2:} Consider a general Ornstein-Uhlenbeck process on $\R^d$, namely a diffusion on $\R^d$ with generator
\begin{eqnarray}\label{EqOUgenerateur}
Lf(x) & = & - (Bx). \na f(x) + \text{div} \po D \na f\pf(x)
\end{eqnarray}
where $B,D\in\mathcal M_{d\times d}( \R)$, $D$ is symmetric positive semidefinite. When ker$D$ does not contain any non-trivial subspace which is invariant by $B$, the process is hypoelliptic (cf. \cite[p. 148]{Hormander1967}). The carr\'e du champ operator is $\Gamma f = (\na f)^T D \na f$ and if $D$ is not definite positive the curvature condition $\Gamma_2 \geq \rho \Gamma$ cannot be fulfilled with $\rho>0$: it would imply via Lemma \ref{LemInegalitePoincare} a log-Sobolev inequality 
\begin{eqnarray*}
\text{Ent}_\mu f & \leq & \frac1{2\rho} \int \frac{\Gamma f}f \dd \mu,
\end{eqnarray*}
which is impossible since the right hand side vanishes for some non-constant $f$, for which the entropy is positive.

As Arnold and Erb proved in \cite{AntonErb}, and as we will see in a revisited version in Section~\ref{SectionGeneral}, even if $D$ is not definite positive, if all the eigenvalues of $B$ lies in $\{\lambda \in \mathbb C, \Re(\lambda)> \rho\}$ for some $\rho>0$ then the process is ergodic and there exists an explicit constant $c$ such that
\begin{eqnarray*}
\frac{|\na P_t f|^2}{P_t f} & \leq & ce^{-\rho t} P_t\po \frac{|\na f|^2}{f}\pf.
\end{eqnarray*}
For $\Phi_1(f) = f\ln f$ it is classical that $\Gamma_{\Phi_1} = \frac{(\na f)^T D \na f}{2f}$ and thus Lemma \ref{LemInegalitePoincare} yields the log-Sobolev inequality
\[\text{Ent}_\mu f  \leq \frac{c|D|}{2\rho} \int \frac{|\na f|^2}{f} \dd \mu. \]
Of course, in this example the invariant law $\mu$ is Gaussian and explicit (\cite[Theorem 3.1]{AntonErb}), and all Gaussian laws satisfy such an inequality. But the arguments extend to the case 
\begin{eqnarray}
Lf(x) & = & - b(x). \na f(x) + \text{div} \po D \na f\pf(x)
\end{eqnarray}
if we simply assume the Jacobian matrix of $b$ is such that $J_b \geq B$ for some $B\in\mathcal M_{d\times d}(\R)$. 
 See Section~\ref{SectionGraph} for an application.

\bigskip

For another example of the link between functional inequalities and $\Gamma$ calculus, see the proof of the H\"older Inequality in the Preface of \cite{BakryGentilLedoux}. Or, actually, see the whole book.

\subsubsection{Hypocoercivity}

Some functional inequality such as those proved in Lemma \ref{LemInegalitePoincare} may also be known a priori, which leads to a third way to use $\Gamma$ calculus:
\begin{lem}\label{LemAvecPoincareApriori}
Suppose $P_t$ is ergodic and its invariant measure satisfies
\begin{eqnarray*}
\forall f\in\A_+\hspace{25pt} 0\hspace{7pt} \leq\hspace{7pt} \int \Phi_1(f) \dd \mu - \Phi_1 \po \int f \dd \mu \pf & \leq & c \int \Phi_2(f) \dd \mu.
\end{eqnarray*}
If moreover $\Gamma_{\Phi_2} \geq \rho \Phi_2 - \beta \Gamma_{\Phi_1}$ for some $\beta>0$ then, writing
\[W(t) = \beta\po \int \Phi_1(P_t) \dd \mu - \Phi_1 \po \int f \dd \mu \pf\pf + \int \Phi_2(P_t f) \dd \mu,\]
we have $W(t) \leq e^{-\frac{2\rho t}{1+\beta c}}W(0)$ and in particular
\begin{eqnarray*}
\int \Phi_2(P_t f)\dd \mu & \leq & (1+\beta c) e^{-\frac{2\rho t}{1+\beta c}} \int \Phi_2(f) \dd \mu.
\end{eqnarray*}
\end{lem}
\begin{proof}
Since $\mu$ is invariant, $\int L\Phi_i(f) \dd \mu = 0$ for all $f$ and $i=1,2$, and
\begin{eqnarray*}
W'(t) & = & -2\int \po\Gamma_{\Phi_2} + \beta \Gamma_{\Phi_1}\pf(P_t f) \dd \mu  \\
& \leq & -2\rho \int \Phi_2(P_t f) \dd \mu\\
& \leq & -\frac{2 \rho}{1+\beta c}W(t).
\end{eqnarray*}
By the Gronwall's Lemma,
\begin{eqnarray*}
 \int \Phi_2(P_t f) \dd \mu & \leq & W(t) \\
& \leq  & W(0) e^{-\frac{\rho t}{1+\beta c}}\\
& \leq & (1+\beta c) e^{-\frac{\rho t}{1+\beta c}}  \int \Phi_2(f) \dd \mu.
\end{eqnarray*}
\end{proof}

\textbf{Example 3:} Consider the kinetic Fokker-Planck diffusion on $\R^{2d}$ with generator
\begin{eqnarray}\label{EqGenekFP}
Lf(x,y) & = & \co y.\na_x - \po y  + \frac{1}{\varepsilon}\na_x U\pf.\na_y + \Delta_y\cf f(x,y),
\end{eqnarray}
where $\varepsilon>0$ and $U\in\mathcal C^\infty(\R^d)$ is quadratic at infinity (cf. \cite{MonmarcheRecuitHypo} for precise assumptions on $U$ and proofs of the upcoming assertions ; see also Lemma \ref{LemGammaEntropic} below). The carr\'e du champs is $\Gamma(f) = |\na_y f|^2$.

The invariant law is $\mu = e^{-\frac{U(x)}{\varepsilon} - \frac{|y|^2}{2}} \dd x \dd y$, the Gibbs law associated to the Hamiltonian $H(x,y) = \frac{U(x)}{\varepsilon} + \frac{|y|^2}{2}$. It satisfies a log-Sobolev inequality
\begin{eqnarray*}
\text{Ent}_\mu(f) & \leq & c_\varepsilon \int \frac{|\na f|^2}{f} \dd \mu,
\end{eqnarray*}
with $c_\varepsilon$ such that $\underset{\varepsilon\rightarrow0}\lim\ \varepsilon \ln c_\varepsilon = d_*>0$ is the critical depth of the potential $U$. Let
\[M=\begin{pmatrix}
1 & -1 \\ -1 & 2
\end{pmatrix},\hspace{35pt} \Phi_2(f) = \frac{(\na f)^T M \na f}{f} = \frac{|(\na_x-\na_y)f|^2 + |\na_y f|^2}{f}.\]
Under the assumption that the Hessian $\na_x^2 U$ is bounded, there exist $\rho,\beta>0$ such that
\begin{eqnarray*}
\Gamma_{\Phi_2}(f) & \geq & \rho \Phi_2(f) - \beta \frac{|\na_y f|^2}{f}.
\end{eqnarray*}
Applying Lemma \ref{LemAvecPoincareApriori} with $\Phi_1(f) = \text{Ent}_\mu(f)$, we obtain
\begin{eqnarray*}
\text{Ent}_\mu(P_t f) &\leq & e^{-\lambda_\varepsilon t} \po \frac3\beta\int \frac{|\na f|^2}{f} \dd \mu + \text{Ent}_\mu(f) \dd \mu\pf
\end{eqnarray*}
with $\lambda_\varepsilon = \frac{2\rho }{1+\beta c_\varepsilon}$. Note that $\beta$ and $\rho$ depends on $\varepsilon$, but not too much, in the sense we still have $\underset{\varepsilon\rightarrow0}\lim\  \varepsilon \ln \lambda_\varepsilon = - d_*$ (see \cite{MonmarcheRecuitHypo} for the consequences on the simulated annealing algorithm based on the kinetic Fokker-Planck dynamics).

\bigskip

This argument is due to Villani \cite{Villani2009} (his idea to work with twisted gradients of the form $(\na f)^T M \na f$ having been inspired by the similar twisted metric of Talay \cite{Talay}, used in a Meyn-Tweedie approach rather than an entropy one) and was initially applied to the kinetic Fokker-Planck process. This case has already been written in "$\Gamma$ settings" (meaning, considering pointwise inequalities and quantities of the form $P_s \Phi(P_{t-s} f)$ rather than $\int \Phi(P_t f) \dd \mu$) by Baudoin in \cite{Baudoin}. Later, Dobeault, Mouhot and Schmeiser proposed in \cite{DMS2009} a similar strategy in $L^2$ with no space derivatives involved, which have from then proved efficient in many contexts.

These methods yield convergences of the form $V(t) \leq c e^{-\rho t} V(0)$ with $c>1$, which are called hypocoercive: the main differences with the case $c=1$ is that hypocoercivity is stable by equivalence (in the sense if $c_1 V(t) \leq W(t) \leq c_2 V(t)$ with $c_1,c_2>0$, then $W$ is also hypocoercive), and is not equivalent to the functional inequality \emph{"$V'(0) \leq - \rho V(0)$ for all $f\in\A$"} (such as the Poincar\'e inequality is equivalent to the exponential decay of the variance, and the log-Sobolev inequality is equivalent to the exponential decay of the entropy; see~\cite{BolleyGentil}).

\bigskip

For another application of Lemma \ref{LemAvecPoincareApriori}, we refer to the study in \cite{MonmarchePDMP} of the TCP with generator
\[Lf(x) = f'(x) + \lambda(x) \po f(hx) - f(x)\pf,\]
where $h\in[0,1)$ and $\lambda$ is a positive increasing function. Note this process is not a diffusion, and the invariant law $\mu$ is not explicit. Nevertheless it is possible to prove $\mu$ satisfies a log-Sobolev inequality, and the $\mathcal H^1$ norm $\int\po P_t f - \int f\dd\mu \pf^2 + |\na P_t f|^2 \dd \mu$ decays exponentially fast with an explicit rate. 

\subsubsection{Global hypoellipticity}

Lemmas \ref{LemGammaCurvature}, \ref{LemInegalitePoincare} and \ref{LemAvecPoincareApriori} are mostly related to the long-time behaviour of the semi-group, but $\Gamma$ calculus may also reveal short-time regularization properties:
\begin{lem}\label{LemRegularisation}
For some $t_0\in(0,\infty]$, let $\Phi_t(f)$ be defined for $t\in[0,t_0]$, $f\in\A_+$ and suppose that $\partial_t \Phi_t(f)$ exists and is in $\A$ for all $f\in \A_+$.
\begin{enumerate}[a)]
\item It $2\Gamma_{\Phi_t} f \geq \partial_t \Phi_t f$ for all $t\in[0,t_0],f\in\A_+$, then
\begin{eqnarray}
\Phi_t(P_t f) & \leq & P_t (\Phi_0 f).
\end{eqnarray}
\item If $P_t$ admits an invariant law $\mu$ such that $\A \subset L^2(\mu)$ and  $2\int \Gamma_{\Phi_t} f \dd \mu \geq \int \partial_t \Phi_t f \dd \mu$ for all $t\in[0,t_0],f\in\A_+$, then
\begin{eqnarray}
\int \Phi_t(P_t f) \dd \mu & \leq &  \int \Phi_0 f \dd \mu.
\end{eqnarray}
\end{enumerate}
\end{lem}
\begin{proof}
Let $\psi(s) = P_s \Phi_{t-s}(P_{t-s} f)$, so that $\psi'(s) = P_s \po 2\Gamma_{\Phi_{t-s}} - \partial_t \Phi_{t-s}\pf(P_{t-s} f)$. In the first case $\psi'(s) \geq 0$, and in the second one $\int \psi'(s) \dd\mu \geq 0$, 
 which conclude in both cases.
\end{proof}

\textbf{Example 4:} this idea was introduced by H\'erau in \cite[Section 3]{Herau2007} (see also \cite[A.21 p.145]{Villani2009}) to study the kinetic Fokker-Planck with generator \eqref{EqGenekFP} via
\begin{eqnarray*}
\Phi_t(f) & =& f^2 +  t \po a|\na_y f|^2 + b t|\na_y f + ct \na_x f|^2\pf
\end{eqnarray*}
where $a$, $b$, $c\in\R$. Under the assumption that the Hessian $\na^2 U$ is bounded (see also \cite[Theorem A.8, A.15 and Remark p.152]{Villani2009} for a proof under the weaker assumption $|\na^2 U | \leq C(1+|\na U|)$) a careful choice of $a$, $b$ and $c$ leads to $\partial_t \int \Phi_t(P_t f) \dd \mu \leq 0$. In particular since $ \int (P_t f)^2 \dd \mu \leq \int f^2 \dd \mu$, denoting by $\|.\|$ the $L^2(\mu)$ norm,
\begin{eqnarray*}
\|\na P_t f\|^2  & \leq & C\frac{\| f -\mu f\|^2 }{t^3}
\end{eqnarray*}
for $t\leq t_0$ for some $C,t_0>0$.

\bigskip

Lemma \ref{LemRegularisation} is reminiscent of the so-called reversed local entropy inequalities (see \cite[Theorem 2]{BolleyGentil}). In fact, suppose $\Phi$ satisfies the generalized $\Gamma_2$ condition $\Gamma_{\Gamma_\Phi} f \geq \rho \Gamma_\Phi$, $\rho \in \R$. Let $c(t)= \frac{e^{2\rho t}-1}{\rho} $ if $\rho \neq 0$ and $c(t) = 2t$ if $\rho = 0$ be the solution of $c'(t) = 2\po 1 + \rho c(t) \pf$ with $c(0)=0$, and
\begin{eqnarray*}
\Phi_t & =& \Phi + c(t)\Gamma_{\Phi}.
\end{eqnarray*}
Then $2\Gamma_{\Phi_t } \geq \partial_t \Phi_t$, hence for all $t\geq 0$
\begin{eqnarray}
P_t \Phi(f) - \Phi(P_t f) & \geq & c(t) \Gamma_{\Phi} (P_t f).  
\end{eqnarray}
\subsection{Quadratic $\Phi$'s, gradient bounds and couplings}

Consider the case of $\Phi(f) = C_1f.C_2f$ when $C_i = (c_{i,1},\dots,c_{i,r})$ is a linear operator from $\A$ to $\A^r$ for some $r\geq 1$. These $\Phi$'s play a particular role as they behave well with any Markov generator and not only with diffusion ones.  The test functions $f$ do not need to be positive.


By convention we write $\Gamma\po C_1 f, C_2 f\pf = \sum_{i=1}^r \Gamma\po c_{1,i}f,c_{2,i} f\pf$ and $[L,C_i] = \po [L,c_{i,1}],\dots,[L,c_{i,r}]\pf$, where $[L,c_{i,j}] = Lc_{i,j} - c_{i,j}L$ is the commutator of $L$ and $c_{i,j}$.

\begin{lem}\label{LemGammaQuadratic}
If $\Phi(f) = C_1f.C_2f$ then
\begin{eqnarray*}
\Gamma_{\Phi}(f) & = & \Gamma(C_1 f, C_2 f) + \frac12 C_1f.[L,C_2]f + \frac{1}{2}[L,C_1]f.C_2f
\end{eqnarray*}
\end{lem}
\begin{proof}
This is a straightforward consequence of the definitions, since $\dd\Phi(f)g = C_1 f.C_2 g + C_1g . C_2 f$.
\end{proof}
In particular, if $C_1 = C_2 = C$, since $\Gamma$ is a positive quadratic operator, we always have
\begin{eqnarray}\label{EqGammaQuadratic}
\Gamma_{|C .|^2} f & \geq & Cf.[L,Cf].
\end{eqnarray}

\textbf{Example 5:} Let $L f(x) = b(x).\na f(x)$ be a deterministic transport operator on $\R^d$. Then $\Gamma = 0$ and
\begin{eqnarray*}
\Gamma_{|\na .|^2} f & = & (\na f)^T [b.\na,\na] f\\
& = & -(\na f)^T J_b \na f.
\end{eqnarray*}
If $J_b(x) \leq -\rho$ for all $x\in\R^d$ with $\rho\in\R$, from Lemma~\ref{LemGammaCurvature},
\[|\na P_t f|^2 \leq e^{-2\rho t} P_t |\na f|^2.\] 
On the other hand if $J_b(x) \geq \rho$ for all $x\in\R^d$ we may invert the signs in the proof of Lemma~\ref{LemGammaCurvature} to get
\[|\na P_t f|^2 \geq e^{2\rho t} P_t |\na f|^2.\] 
Here of course $P_t f(x) = f\po \varphi_x(t)\pf$ where $\varphi$ is the flow solution of
\[\varphi_x(0) = x \hspace{30pt} \varphi'_x(t) = b\po\varphi_x(t)\pf.\]
The Jacobian matrix measures how two trajectories of the deterministic flow tends to get closer or to drift apart. Indeed,
\begin{eqnarray*}
\varphi_x(t) - \varphi_y(t) &  =&  x- y + t J_b(x)(x-y) + t\underset{y\rightarrow x}o(x-y) + \underset{t\rightarrow 0}o(t) \\
\Rightarrow\hspace{15pt}|\varphi_x(t) - \varphi_y(t)|^2 & = & |x-y|^2 + 2 t(x-y)^*J_b(x)(x-y)+ t\underset{y\rightarrow x}o\po|x-y|^2\pf + |x-y| \underset{t\rightarrow 0}o(t).
\end{eqnarray*}
A negative Jacobian means the flow locally contracts the space in the vicinity of all points.

\bigskip

In the classical Bakry-\'Emery theory, $\Phi f = \Gamma f$ is a natural choice in particular because it only depends on the Markov dynamics (it is, in a sense, intrinsic). On the other hand, the choice $\Phi f = |\na_d f|^2$, where $\na_d$ is the (upper-)gradient associated to the metric $d$ on $E$ by
\begin{eqnarray}\label{EqdefGradWasser}
|\na_d f|(x) & := & \underset{r\rightarrow0}\lim \underset{0<d'x,y)\leq r} \sup \frac{|f(x)-f(y)|}{d(x,y)},
\end{eqnarray}
 is linked to the Wasserstein distances on $\mathcal P(E)$, as shown by Kuwada \cite{Kuwada,Kuwada2013} (see also \cite{CattiauxGuillin2014}). Recall the Wasserstein distance $\mathcal W_{p}$ on $\mathcal P(E)$ is defined by
\begin{eqnarray*}
\mathcal W_p(\mu,\nu) & = & \inf \{\po\mathbb E\po d^p(U,V)\pf\pf^{\frac1p},\ \text{law}(U)= \mu,\text{law}(V)=\nu\}.
\end{eqnarray*}
When the marginal laws of $(U,V)$ are $\mu$ and $\nu$, $(U,V)$ is called a coupling of $\mu$ and $\nu$. For $\mu \in \mathcal P(E)$ we write $\mu Q$ the image of $\mu$ by a Markov operator $Q$, so that $(\mu Q) f = \mu (Qf)$.

 We state the following result only in $\R^d$ endowed with the Euclidian metric and refer to \cite{Kuwada,Kuwada2013} for more general settings:
\begin{thm}[Theorem 2.2 of \cite{Kuwada}]\label{TheorKuwada} 
Let $P$ be a Markov semi-group on $\R^d$ and let $p>1$ and $q = \frac{p}{p-1}$ be its H\"older conjugate. Then
\begin{eqnarray}\label{EqContractWasserstein}
 \forall \mu,\nu \in \mathcal P(\R^d),\hspace{25pt}\mathcal W_p\po \mu P,\nu P\pf & \leq & \gamma \mathcal W_p\po \mu ,\nu \pf
\end{eqnarray}
for some $\gamma>0$ implies
\begin{eqnarray}\label{EqContractGradient}
\forall f\in\mathcal C_{b,L}, x\in\R^d,\hspace{25pt} |\na P f|(x) & \leq & \gamma \po P\po |\na f|^q \pf\pf^{\frac1q}(x)
\end{eqnarray}
where $\mathcal C_{b,L}$ is the set of bounded continuous Lipschitz functions on $\R^d$.

Moreover if for all $x\in\R^d$, $P(x)$ admits a continuous density with respect to the Lebesgue measure, then the converse is true: \eqref{EqContractGradient} implies \eqref{EqContractWasserstein}.
\end{thm}

\textbf{Example 6:} Consider the jump operator $Lf(x) = \lambda \po Qf(x) - f(x) \pf$ on $\R^d$ where $Qf(x)$ is a Markov operator and $\lambda >0$ is constant. If $(Y_k)_{k\geq 0}$ is a Markov chain with operator $Q$, in the sense $\mathbb E\po f(Y_{k+1}) | Y_k = y\pf = Q f(y)$, and $(N_t)_{t\geq 0}$ is a Poisson process with intensity $\lambda$, then $X_t = Y_{N_t}$ is a Markov process with generator $L$. It is possible to define two processes $X$ and $X'$ with generator $L$ starting at different points $x$ and $x'$ that will always jump together, by considering the same Poisson process, so that the way both processes may get closer or drift apart is given by the behaviour of the associated Markov chains $Y$ and $Y'$. Suppose $Q$ is such that one can define a coupling $(Y,Y')$ where both $Y$ and $Y'$ are Markov chains associated to $Q$ and
\begin{eqnarray*}
\mathbb E\po |Y_{k+1} - Y'_{k+1} |^2 \ | \ Y_k=y,Y_k'=y'\pf & \leq & \gamma |y - y'|^2.
\end{eqnarray*}
for some $\gamma>0$. If $\gamma<1$ the space is contracted by the jumps, if $\gamma>1$ it may be expanded. In any case
\begin{eqnarray*}
\mathbb E\po |X_t - X_t|^2\pf & \leq & \mathbb E\po \gamma^{N_t}\pf |x-x'|^2\\
& = & e^{\lambda(\gamma - 1) t} |x-x'|^2.
\end{eqnarray*}
The above couplings prove that for all $\mu,\nu \in \mathcal P(\R^d)$,
\begin{eqnarray*}
\mathcal W_2^2(\mu Q,\nu Q) & \leq & \gamma \mathcal W_2^2(\mu,\nu ) \\
\mathcal W_2^2(\mu P_t,\nu P_t)  & \leq & e^{-\lambda(\gamma-1) t}\mathcal W_2^2(\mu ,\nu) 
\end{eqnarray*}
which implies by Theorem \ref{TheorKuwada} that for all $f\in \mathcal C_{b,L}$
\begin{eqnarray}\label{EqSautGrad1}
|\na Q f|^2 & \leq & \gamma Q|\na f|^2 \\\label{EqSautGrad2}
|\na P_t f|^2  & \leq & e^{-\lambda(\gamma-1) t}P_t|\na f|^2. 
\end{eqnarray}
With $\Gamma$ calculus, we directly obtain \eqref{EqSautGrad2} from \eqref{EqSautGrad1}:
\begin{eqnarray*}
\Gamma_{L,|\na .|^2}f & = & \frac{\lambda}{2}\po Q|\na f|^2 + |\na f|^2 - 2 \po\na f\pf^T \po\na Qf\pf\pf\\
& \geq & \frac{\lambda}{2}\po Q|\na f|^2 + |\na f|^2 - 2  |\na f | \sqrt{\gamma Q|\na f|^2}\pf \hspace{30pt}\text{given \eqref{EqSautGrad1}}\\
& \geq & \frac{\lambda(1-\gamma)}{2}|\na f|^2
\end{eqnarray*}
and Lemma \ref{LemGammaCurvature} yields \eqref{EqSautGrad2}. Note that $P_t(x)$ is not absolutely continuous with respect to the Lebesgue measure, since $\mathbb P\po X_t = x\ | \ X_0 = x\pf = e^{-\lambda t}$, and thus we cannot retrieve the Wasserstein contraction from Theorem \ref{TheorKuwada}.

\bigskip

For an example where the metric is not the Euclidian one, we refer to the study of the TCP with linear rate on $\R_+$, with generator
\[Lf(x) = f'(x) + x \po f\po hx\pf - f(x)\pf\]
where $h\in[0,1)$. A coupling approach may be found in \cite{Malrieu2011}, and a $\Gamma$ one in \cite[Section 4.4]{MonmarchePDMP}.

\subsection{Diffusions and entropic $\Phi$'s}\label{SubsectionEntropic}
In this subsection we suppose $\A$ is fixed by the composition by any smooth 
compactly supported $a\in\mathcal C^\infty_{c}(\R)$, and by exponentiation $f \rightarrow f^\alpha$ for all $\alpha \geq 0$.

We say $L$ is a diffusion operator (or equivalently $(P_t)_{t\geq0}$ is a diffusion semi-group or $(X_t)_{t\geq0}$ a diffusion process) if one of the following equivalent conditions holds:
\begin{enumerate}[(i)]
\item for all $f,g\in\A$,
\begin{eqnarray*}
L(fg) & = & gLf + fLg + 2\Gamma(f,g)
\end{eqnarray*}
\item for all $a\in\mathcal C^\infty_c$, $f\in\A$,
\begin{eqnarray*}
L\po a(f) \pf(x) & = & a'\po f(x)\pf Lf(x) + a''\po f(x)\pf \Gamma f(x).
\end{eqnarray*}
\item for all $a\in\mathcal C^\infty_c$, $f,g\in\A$
\begin{eqnarray*}
\Gamma\po a(f), g\pf & =& a'(f) \Gamma(f,g)
\end{eqnarray*}
\end{enumerate}
When $L$ is a diffusion then conditions (i) and (ii) hold for any $a$ such that the terms are well-defined for $f\in \A$. 
As a direct consequence of the condition (ii), if $\Phi f(x) = a\po f(x) \pf$ and $L$ is a diffusion operator, $\Gamma_{\Phi} f = \frac{a''(f)}2 \Gamma(f)$. More generally,
\begin{lem}\label{LemGammaDePuissance}
If $L$ is a diffusion operator and $\Phi_2(f)(x) = a\po \Phi_1 f(x)\pf$ then
\begin{eqnarray*}
\Gamma_{\Phi_2} f & = & a'(\Phi_1 f) \Gamma_{\Phi_1} f + \frac{a''\po \Phi_1 f\pf}2 \Gamma\po\Phi_1 f\pf
\end{eqnarray*}
\end{lem}
\begin{proof}
It results from the definitions and the fact $\dd \po a(\Phi_1)\pf = a'(\Phi_1) \dd \Phi_1$.
\end{proof}

\textbf{Example 7:} If $\Phi_1\geq 0$ satisfies the curvature condition $\Gamma_{\Phi_1} \geq \rho \Phi_1$ then for all $p>1$, from Lemma \ref{LemGammaDePuissance}, $\Phi_2 = \Phi_1^p$ satisfies $\Gamma_{\Phi_2} \geq p \rho \Phi_2$, which yields $\Phi_1(P_t f) \leq e^{-\rho t} \po P_t (\Phi_1 f)^p\pf^{\frac{1}{p}}$. Of course by Jensen inequality this is weaker than $\Phi_1(P_t f) \leq e^{-\rho t}  P_t \Phi_1 f$. When $\Phi_1 f= |\na f|^2$, considering Theorem \ref{TheorKuwada}, this is consistent with the fact a contraction of $\mathcal W_p$ is stronger than a contraction of $\mathcal W_{p'}$ when $p'<p$.

\bigskip

\begin{defi*}
When $a$ is a convex function on $\R_+$ and $\nu\in\mathcal P(E)$,
\begin{eqnarray*}
\text{\emph{Ent}}_\nu^a (f) & = & \nu \po a(f) \pf - a\po \nu f\pf,
\end{eqnarray*}
which is positive, is called the $a$-entropy of $f$ with respect to $\nu$. For $E=\R^d$ we will say $a \in \mathcal C^4$  is an admissible function if $\frac{1}{a''}$ is positive concave.
\end{defi*}
Note that an admissible function is necessarily strictly convex, and thus $\text{Ent}_\nu^a (f) = 0$ iff  $\nu$-almost everywhere $f=\nu f$. When $a$ is an admissible function, the curvature condition $\Gamma_2 \geq \rho \Gamma$ implies $\Gamma_{\Gamma_{\Phi}} \geq \rho \Gamma_{\Phi}$ with $\Phi f= a(f)$ (we refer to \cite[Theorem 2 and Remark 3]{BolleyGentil} and references within for more considerations on this matter) which, following the proof of Lemma~\ref{LemInegalitePoincare}, yields
\begin{eqnarray*}
\text{Ent}_{P_t(x)}^a(f) & \leq & \frac{1-e^{-2\rho t}}{\rho} P_t\po a''(f) \Gamma(f)\pf(x).
\end{eqnarray*}
When the curvature condition does not hold with $\Gamma_2$, we can try to compute $\Gamma_{\Phi}$ for $\Phi$ of a form similar to $a''(f) \Gamma f$, which leads to the following:

\begin{lem}\label{LemGammaEntropic}
If $L$ is a diffusion operator, $C:\A\rightarrow \A^r$ is a linear operator and $a: \R_+ \rightarrow \R_+$ is an admissible function such that $\Phi f = a''(f)|C f|^2\in \A_+$ for all $f\in \A_+$, then
\begin{eqnarray}
\Gamma_\Phi f& \geq & a''(f)Cf.[L,C]f.
\end{eqnarray} 
\end{lem}
\begin{proof}
Let $\Phi_1 f = | C f|^2$ and $\alpha = \frac1{a''}$. By the diffusion property,
\begin{eqnarray*}
L \po \frac{\Phi_1(f)}{\alpha(f)}\pf & = & \frac{L \Phi_1 f}{\alpha(f)} + \Phi_1 f \co  \po \frac{-\alpha'(f)L f}{\alpha^2(f)}\pf+\po \frac1\alpha\pf'' \Gamma f \cf + 2 \Gamma\po \Phi_1 f, \frac{1}{\alpha(f)}\pf
\end{eqnarray*}
with $\po \frac 1\alpha\pf'' = -\frac{\alpha''}{\alpha^2} + 2\frac{(\alpha')^2}{\alpha^3}\geq 2\frac{(\alpha')^2}{\alpha^3}$. On the other hand
\begin{eqnarray*}
\dd \po \frac{\Phi_1}{\alpha} \pf(f).Lf & = & \frac{\dd \Phi_1(f).Lf}{\alpha(f)} - \Phi_1 f\frac{\alpha'(f) Lf}{\alpha^2(f)}.
\end{eqnarray*}
Therefore
\begin{eqnarray*}
\Gamma_{\Phi} f & \geq & \frac{\Gamma_{\Phi_1}(f)}{\alpha(f)} + \frac{(\alpha')^2(f)}{\alpha^3(f)} \Phi_1 f \Gamma f + \Gamma\po \Phi_1 f, \frac{1}{\alpha(f)}\pf
\end{eqnarray*}
By Lemma \ref{LemGammaQuadratic}, $\Gamma_{\Phi_1}f = \Gamma\po Cf\pf + Cf.[L,C]f$, and by the diffusion property
\[\Gamma\po\Phi_1 f, \frac{1}{\alpha(f)}\pf = \frac{\alpha'(f) 2 C f. \Gamma(Cf,f)}{\alpha^2(f)}\]
 where by convention $\Gamma(Cf,f) = \po \Gamma(c_1 f),\dots, \Gamma(c_r f,f)\pf$. As a positive bilinear form $\Gamma$ satisfies the Cauchy-Schwarz inequality $|\Gamma(Cf,f)|^2 \leq \Gamma(Cf)\Gamma f$, and thus
 \begin{eqnarray*}
\Gamma_{\Phi} f & \geq & \frac{\Gamma(Cf)+ Cf.[L,C]f}{\alpha(f)}  + \frac{(\alpha')^2(f)}{\alpha^3(f)} \Phi_1 f \Gamma f- 2 \sqrt{\frac{(\alpha')^2(f)\Phi_1(f) \Gamma f}{\alpha^3(f)}}\sqrt{\frac{\Gamma(Cf)}{\alpha(f)}}\\
& \geq & \frac{Cf.[L,C]f}{\alpha(f)}.
\end{eqnarray*}  
\end{proof}

\textbf{Example 8:}
Let $b$ be a smooth vector field on $\R^d$, $D = Q^T Q$ be a constant semidefinite positive matrix and
\begin{eqnarray*}
L f & = & b(x).\na f + \text{div}\po D \na f\pf.
\end{eqnarray*}
The carr\'e du champ is $\Gamma f = (\na f)^T D \na f = |Q \na f|^2$. For $M = P^TP$ a constant positive matrix, let
\[\Phi_M(f) = a''(f)|P\na f|^2 = a''(f)(\na f)^T M \na f .\]
Since $[L,P\na] = -PJ_b\na$, Lemma \ref{LemGammaEntropic} yields
\begin{eqnarray}\label{EqGammadeDiffusion}
\Gamma_{\Phi_M}f & \geq &  -a''(f)(\na f)^T M J_b \na f
\end{eqnarray}
This explains the computations of Examples 2 and 3. For Ornstein-Uhlenbeck processes, see also Corollary \ref{CorOU} below. But above all, it is the key ingredient of the next section.

\section{General results}\label{SectionGeneral}


 We propose now to prove the $\Gamma$ counterparts of Villani's general results \cite[Theorem A.15 p.158 and Theorem 28, p.42]{Villani2009} concerning (entropic) hypoellipticity and hypocoercivity for diffusions operators.

Suppose the diffusion generator $L$ on $\R^d$ is on H\"ormander form
\[L = B_0 + \sum_{i=1}^d B_i^2\]
where $B = (B_j)_{0\leq j\leq r} = \po b_j.\na\pf_{0\leq j\leq r}$ is a derivation operator on $\R^d$ with $\mathcal C^\infty$ coefficient $b_j$'s which are not necessarily linearly independent. Let $a$ be an admissible function, that is to say a strictly convex $\mathcal C^4$ function from $\R_+$ to $\R_+$ such that $\frac{1}{a''}$ is concave. 
 Let $\A$ be the set of $\mathcal C^\infty $ functions on $\R^d$ whose all derivatives grow at most as polynomials at infinity. Suppose $\A$ is fixed by $P_t$, $L$ and that if $\A_{p}$ is the set of functions in $\A$ which are bounded by a positive constant, $f\in\A_p$ implies $a(f) \in \A$.  Moreover suppose that if $\mu\in\mathcal P(\R^d)$ is invariant for $P_t$ and have a finite exponential moment then for all positive $f \in L^1(\mu)$ such that $\text{Ent}_\mu^a f < \infty$, there exists a sequence $(f_m)_{m\in\mathbb N}$ of Lipschitz bounded functions in $\A_p$ such that  as $m$ goes to infinity, $f_m \rightarrow f$ in $L^1(\mu)$ and for all $t\geq 0$,  $\text{Ent}_\mu^a P_t f_m \rightarrow \text{Ent}_\mu^a P_t f$ (this assumption is clearly satisfied if $a(f)=f^p$, $p>1$, since  in $L^p(\mu)$, $\A$ is dense and $P_t$ is continuous). 
 
\bigskip 
 
For $i\geq 0$, $n\in \mathbb N$, we write
\begin{eqnarray*}
 C_i & =& (c_{i,1}.\na,\dots,c_{i,n}.\na)\\
R_i & = & (r_{i,1}.\na,\dots,r_{i,n}.\na)
\end{eqnarray*}
where the $r_{i,j}$'s and $c_{i,j}$'s are $\mathcal C^\infty$ vector fields on $\R^d$, and $Z_i = (z_{i,k,l})_{1\leq k,l\leq  n} \in \mathcal M_n(\R)$ where the $z_{i,k,l}$'s are $\mathcal C^\infty$ scalar fields. By convention 
 $R_1^TR_2$ will stand for the quadratic operator $f\mapsto (R_1 f)^TR_2 f$, and 
 recall we set
\begin{eqnarray*}
[L,C_i] & := & \po [L,c_{i,1}],\dots,[L,c_{i,n}]\pf.
\end{eqnarray*}

\begin{thm}\label{TheoGeneralEllipticite}
Suppose there exist $N_c\in \mathbb N$ and $\lambda,\Lambda,m>0$ such that for $i\in\llbracket 0,N_c+1\rrbracket$ there exist derivation operators $C_i$ and $R_i$ and a matrix field $Z_i$, all with coefficients in $\A$, satisfying:
\begin{enumerate}[(i)]
\item $C_{N_c+1} = 0$,\ and \  $[B_0,C_i] = Z_{i+1} C_{i+1} + R_{i+1}$ \  for all $i\in\llbracket 0,N_c\rrbracket$ .
\item $[B_j,C_i] = 0$ \ for all $i\in\llbracket 0,N_c\rrbracket$, $j\in \llbracket 1,d\rrbracket$, 
\item  $\lambda \leq \frac{Z_i+Z_i^T}2 \leq \Lambda$ \ for all $i\in\llbracket 0,N_c\rrbracket$,
\item $C_0^TC_0 \leq m \underset{j\geq 1}\sum B_j^TB_j$ and $R_i^T  R_i \leq m \underset{j<i}\sum C_j^T C_j$ for all $i\in\llbracket 0,N_c+1\rrbracket$.
\end{enumerate}
Then there exists $c>0$ such that for all $f\in\A_p$, $x\in\R^d$,  $t>0$  and $i\in\cco 0,N_c \ccf$
\begin{eqnarray}
a''\po P_t f(x)\pf  |C_i \na P_t f|^2(x) & \leq & c (1-e^{-t})^{-2 i-1} \text{\emph{Ent}}_{P_t(x)}^a f.
\end{eqnarray}
\end{thm}

\begin{thm}\label{TheorGeneralCoercivite}
Under the same assumptions as Theorem \ref{TheoGeneralEllipticite}, if moreover there exist $\rho,K>0$ and a measure $\mu\in\mathcal P(\R^d)$ such that
\begin{enumerate}[(a)]
\item $\underset{i\geq 0}\sum C_i^T C_i \geq \rho$,
\item $\mu$ is invariant for $P_t$ and satisfies the entropic inequality
\begin{eqnarray}\label{EqThoerGeneCoer}
\forall f\in \A_p,\hspace{30pt} \text{\emph{Ent}}_\mu^a f & \leq & K \int a''(f) |\na f|^2 \dd \mu,
\end{eqnarray}
\end{enumerate}
then there exist $\kappa>0$ such that for all $t>0$ and for all $f$ with \emph{Ent}$_\mu^a f <\infty$,
\begin{eqnarray}\label{EqTheorCoerc}
\text{\emph{Ent}}_\mu^a (P_t f) & \leq &  e^{-\kappa t(1-e^{-t})^{2N_c}} \text{\emph{Ent}}_\mu^a f.
\end{eqnarray}
\end{thm}

\textbf{Remark:} Note that these are not exactly the same results as in \cite{Villani2009}: the commutation condition (ii) we require is a very strong assumption (reminiscent of \cite[Remark 33 p. 45]{Villani2009}). It holds if the matrix diffusion and the $C_i$'s are constant, or in dimension 2 if $B=(B_1,0)$ and $C_1 = B_1$, and this is basically all. That being said for most of the models we have in mind (and to which Villani's method have been applied, to our knowledge) the diffusion matrix is indeed constant. 

On the plus side our results do not need to consider adjoint operators in $L^2(\mu)$ where $\mu$ is the invariant measure (which is the case in Villani's work but also for the method of Dolbeault, Mouhot and Schmeiser), which would require to have an explicit expression or at least some informations on the density $\mu(x)$ (see discussion \cite[\S 9.2 p.67]{Villani2009}), or to work with the wrong invariant measure as in \cite{Hoffmann2016}. By contrast, Theorem~\ref{TheoGeneralEllipticite} does not even need an invariant measure, while Theorem~\ref{TheorGeneralCoercivite} only needs its existence and the functional inequality \eqref{EqThoerGeneCoer} (which may be hard to establish, of course, when $\mu$ is not explicit; see \cite{MonmarchePDMP} for such an example, with no hypoellipticity).

\bigskip

Both theorems are based on the following computation:

\begin{lem}\label{LemGeneralCalculPhi}
Let  $\alpha(t) = 1 - e^{-t}$, $(\varepsilon_i)_{i\in \llbracket 0,N_c\rrbracket}\in(0,1)^{N_c+1}$ and
\begin{eqnarray*}
\Phi_{0} f & = & a(f)\  + \ a''(f) \varepsilon_{0}^2 \alpha(t) |C_0 f|^2  \\
\Phi_i f & = & a''(f) \varepsilon_{i} \alpha^{2i-1}(t)|\po C_{i-1} + \varepsilon_i\alpha(t) C_{i}\pf f|^2\hspace{25pt}\text{for }i\in\llbracket 1,N_c\rrbracket.
\end{eqnarray*}
There exist $b_1,b_2,b_3>0$ and $\varepsilon_* \in (0,1)$ such that for all $i\in\llbracket 0,N_c\rrbracket$, if $\varepsilon_i < \varepsilon_*$,
\begin{eqnarray*}
\frac{\po2 \Gamma_{\Phi_i} -  \partial_t \Phi_i\pf f}{a''(f)} & \geq &  - b_1 \po\sum_{j< i} \alpha^{2j}|C_j f|^2\pf + b_2\varepsilon_i^2\alpha^{2i}|C_{i} f|^2 - b_3 \alpha^{2(i+1)}\varepsilon_i^4 |C_{i+1}  f|^2 .
\end{eqnarray*}
\end{lem}
\begin{proof}
For $i=0$, since $\sum_{j<0} |C_j  f|^2 =0$ and $\alpha'(t) \leq 1$, from Lemmas \ref{LemGammaDePuissance} and \ref{LemGammaEntropic},
\begin{eqnarray*} 
\frac{\po2 \Gamma_{\Phi_0} -  \partial_t \Phi_0\pf}{a'' }  & \geq &  \Gamma  + \varepsilon_0^2 C_0^T \po 2 \alpha \po Z_1 C_1 + R_1\pf - \alpha' C_0\pf\\
& \geq & \po\frac1m - \varepsilon_0^2 (2\sqrt m +1)\pf C_0^T C_0 - 2\varepsilon_0^2 \alpha \Lambda C_0^T C_1\\
& \geq & \po\frac1{2m} - \varepsilon_0^2 (2\sqrt m +1)\pf C_0^T C_0 - 2m\varepsilon_0^4 \alpha^2 \Lambda^2 C_1^T C_1.
\end{eqnarray*}
For $i\geq 1$, similarly, 
\begin{eqnarray}\label{EqcalculPhiGeneral}
\frac{\po2 \Gamma_{\Phi_i} -  \partial_t \Phi_i\pf}{a'' }  & \geq &  \varepsilon_i(C_{i-1}+\alpha \varepsilon_i C_{i})^T\Big[2\alpha^{2i-1}(Z_{i} C_{i} + R_{i})  - 2\alpha^{2i-1} \alpha'\varepsilon_i C_{i}\notag\\
& & + 2\alpha^{2i}\varepsilon_i(Z_{i+1} C_{i+1} + R_{i+1}) - (2i-1)\alpha'\alpha^{2i-2}(C_{i-1}+\alpha\varepsilon_i C_{i})\Big].
\end{eqnarray}
In order to describe how we bound this expression, the following array $\{ q_{i,j}, i=1,2,\ j=1..5\}$ is built as follow: at line $i$ and column $j$, $q_{i,j}$ is the coefficient in the left-hand side of \eqref{EqcalculPhiGeneral} of the product $c_j^Tl_i$ (the coefficient of $C_i^T C_{i-1}=C_{i-1}^TC_i$ is distributed between $c_1^T l_2$ an $c_2^T l_1$). Moreover the decomposition $q_{i,j} = 2 p_{i,j} \times p'_{i,j}$ with the $\times$ sign means that we bound $ q_{i,j} (c_j f)^T(l_i f)   \geq -|p_{i,j} l_i  f|^2 - |p'_{i,j} c_j  f|^2$ (the left term goes with $l_i$, the right one with $c_j$).
\begin{displaymath}
\begin{array}{|l|c|c|}
\hline
 & c_1 = C_{i-1} & c_2 = C_{i}\\
 \hline
 & & \\
 l_1 = C_{i-1} & -\varepsilon_i(2i-1)\alpha'\alpha^{2i-2} 	&  - 2\po (2i-1) \alpha' \alpha^{i-2}\pf \times\po\varepsilon_i^2 \alpha^{i}\pf \\
 & & \\
 l_2 = C_{i} &  2 \po  \sqrt{\frac{ \lambda}{2}}\alpha^i \varepsilon_i  \pf\times \po \sqrt{\frac{2}{\lambda}} \alpha^{i-1}(Z_i-\varepsilon_i \alpha')\pf &  \alpha^{2i}\varepsilon_i^2 \po  2Z_{i} - (2i+1)\alpha'\varepsilon_i \pf  \\
 & & \\
 l_3 = C_{i+1} & 2 \po \alpha^{i+1} \varepsilon_i^{2} \pf \times \po\alpha^{i-1} Z_{i+1} \pf & 2  \po \sqrt{\frac{2 }{\lambda}}Z_{i+1} \varepsilon_i^2 \alpha^{i+1}\pf \times \po  \sqrt{\frac{ \lambda}{2}} \varepsilon_i \alpha^i\pf\\
 & & \\
 l_4 = R_{i} & 2 \po \varepsilon_i \alpha^{i} \pf \times \po \alpha^{i-1}\pf  & 2\po \alpha^{i}\pf \times \po \alpha^{i} \varepsilon_i^2\pf\\
 & & \\
 l_5 = R_{i+1} & 2\po \varepsilon_i^2 \alpha^{i}\pf\times\po\alpha^{i}\pf & 2\po\alpha^{i+1}\varepsilon_i^{\frac32}\pf \times \po\alpha^{i}\varepsilon_i^{\frac32}\pf\\
 & & \\
 \hline
\end{array}
\end{displaymath}
For instance (line 3 column 1) in \eqref{EqcalculPhiGeneral} appears the term $2 \alpha^{2i} \varepsilon_i^2 (C_{i-1} f)^T Z_{i+1} C_{i+1}f$, which is bounded below by $-\alpha^{2(i+1)} \varepsilon_i^4|C_{i+1} f|^2 - \alpha^{2(i-1)}\Lambda^2 |C_{i-1} f|^2$.

From the operation presented via the array, together with $\alpha,\alpha',\varepsilon_i\leq 1$, $\lambda \leq \frac{Z_i+Z_i^T}2 \leq \Lambda$, $|R_i f|^2 \leq m \sum_{j< i} |C_j  f|^2$ and  $|R_{i+1}\na f|^2 \leq m \po |C_{i}f|^2 + \sum_{j< i} |C_j  f|^2\pf$, we obtain
\begin{eqnarray*}
\frac{\po2 \Gamma_{\Phi_i} -  \partial_t \Phi_i\pf}{a''(f)}  & \geq &  -d_1 \alpha^{2(i-1)} \sum_{j< i} |C_j f|^2 + d_2\alpha^{2i} \varepsilon_i^2 |C_i f|^2 - d_3 \varepsilon_i^4 \alpha^{2(i+1)} |C_{i+1} f|^2
\end{eqnarray*}
with some $d_1,d_3$ and $d_2 = \lambda\po 2 - \frac12 - \frac12\pf + \underset{\varepsilon_i \rightarrow 0}{\mathcal O}(\varepsilon_i)$ where the negligible term is uniform with respect to $i \in \cco 1,N_c \ccf$, which concludes.
\end{proof}

\begin{proof}[Proof of Theorem \ref{TheoGeneralEllipticite}]
Keep the notations of Lemma \ref{LemGeneralCalculPhi}, and let $\varepsilon_{N_c} = \frac12 \min\po \varepsilon_*, \frac{b_2}{2b_1 + b_3}\pf$, $\varepsilon_{i-1} = \varepsilon_{i}^3$ for $i\leq N_c$, $\lambda_{0} = 1$, and $\lambda_{i+1} = \varepsilon_i^3\lambda_i$ for $i\geq 0$. Note in particular that 
\[\sum_{j>i} \lambda_j \leq \lambda_{i+1} \sum_{k\geq 0} \frac1{2^k} = 2 \varepsilon_i^3 \lambda_{i}.\]
 Let
\begin{eqnarray*}
\Phi_{(t)} f & = & \sum_{i=0}^{N_c} \lambda_i \Phi_i f,
\end{eqnarray*}
so that
\begin{eqnarray*}
\frac{2\Gamma_{\Phi_{(t)}} - \partial_t \Phi_{(t)}}{a''(f)} & \geq & \sum_{i=0}^{N_c} \alpha^{2i} |C_i  f|^2 \po b_2 \lambda_i \varepsilon_i^2 - b_1 \po \sum_{j>i} \lambda_j \pf - b_3 \lambda_{i-1} \varepsilon_{i-1}^4 \pf\\
& \geq & \sum_{i=0}^{N_c} \alpha^{2i} |C_i  f|^2 \lambda_i\po b_2 \varepsilon_i^2 - (2 b_1 + b_3)\varepsilon_i^3 \pf\\
& \geq & \frac{b_2}2 \sum_{i=0}^{N_c} \alpha^{2i} |C_i  f|^2 \lambda_i \varepsilon_i^2  \hspace{15pt}>0.
\end{eqnarray*}
It means $\psi(s) = P_s \Phi_{(t-s)} (P_{t-s} f)$ is increasing, and $\psi(t) \geq \psi(0)$ reads
\begin{eqnarray*}
 P_t a(f) - a(P_t f)  & \geq &  \varepsilon_0^2 \alpha(t) |C_0 P_t f|^2 + \frac1{\lambda_0 }\sum_{i\geq 1} \lambda_i \Phi_{i} (P_t f) \\
& \geq & \frac{1}{c}  \sum_{i=0}^{N_c} (1-e^{-t})^{2i+1}|C_i  P_t f|^2
\end{eqnarray*}
for some $c>0$.
\end{proof}

\begin{proof}[Proof of Theorem \ref{TheorGeneralCoercivite}]
With $\Phi_{(t)}$ defined above and the new assumptions,
\begin{eqnarray*}
\frac{2\Gamma_{\Phi_{(t)}} - \partial_t \Phi_{(t)}}{a''(f)}  & \geq & \rho_1 (1-e^{-t})^{2N_c}   \po |\na  f|^2 + \sum_{i=0}^{N_c} |C_i   f|^2 \pf
\end{eqnarray*}
for some $\rho_1>0$. It means, writing $W(t) =  \frac1{\lambda_0} \int \Phi_{(t)} (P_t f) \dd \mu - a\po \int f\dd\mu\pf$, that
\begin{eqnarray*}
W'(t) & \leq & -  \frac{\rho_1}{\lambda_0}  (1-e^{-t})^{2N_c}  \int  a''(P_t f) \po|\na P_t f|^2   + \sum_{i=0}^{N_c} |C_i  P_t f|^2 \pf \dd \mu\\
& \leq & - \rho_2 (1-e^{-t})^{2N_c} W(t) 
\end{eqnarray*}
for some $\rho_2$ thanks to Inequality \eqref{EqThoerGeneCoer}. Therefore
\begin{eqnarray*}
\text{Ent}_\mu^a (P_t f) & \leq &  W(t) \\
& \leq & e^{-\rho_2\int_0^t (1-e^{-s})^{2N_c} \dd s } W_0\\
& = & e^{-\rho_2\int_0^t (1-e^{-s})^{2N_c} \dd s } \text{Ent}_\mu^a  f.
\end{eqnarray*}
A priori the result holds for $f\in\A_p$, but it does not depend on the regularity of $f$ nor on its positive bound and we conclude by a density argument (indeed Inequality \eqref{EqThoerGeneCoer} implies $\mu$ satisfies a Poincar\'e inequality, which is the case $a(f) = f^2$, which in turn implies $\mu$ admits a finite exponential moment). 
\end{proof}

\textbf{Remark:} a thorough reading of the proofs show that, if $\rho,\lambda \leq 1 \leq \Lambda,K,m$ (which can always be assumed), one can choose $\varepsilon_* = \frac{\lambda}{10mN_c}$, $b_1 = 7 \po N_c^2 +\frac{\Lambda^2}{\lambda} + m\pf$, $b_2 =\frac{\lambda}{2}$ and $b_3 = \frac{3\Lambda^2}{\lambda}$ in Lemma \ref{LemGeneralCalculPhi}, and the constants $c$ and $\kappa$ respectively in Theorems \ref{TheoGeneralEllipticite} and \ref{TheorGeneralCoercivite} may be chosen has
\begin{eqnarray*}
c_* & =& \po \frac{100}{\lambda}\po N_c^2 + \frac{\Lambda^2}{\lambda}+m\pf \pf^{20 N_c^2}\\
\kappa_* &=& \frac{\rho}{c_*K}
\end{eqnarray*}
For instance, consider a family $(L_\varepsilon)_{\varepsilon>0}$ of generators such that Theorems \ref{TheoGeneralEllipticite} and \ref{TheorGeneralCoercivite} holds with $N_c$ uniform w.r.t $\varepsilon$, constants $\Lambda_\varepsilon,m_\varepsilon,\lambda_\varepsilon^{-1},\rho_\varepsilon^{-1}$ that grow at most polynomially and $\varepsilon \ln K_\varepsilon \rightarrow E > 0$ as $\varepsilon \rightarrow 0$,  (which is typically the case in a metastable context). Then $\varepsilon \ln \kappa_\varepsilon \rightarrow -E$, or in other words the speed of convergence for small $\varepsilon$ (in a large deviation scaling) is given by the degree of metastability, just as in the classical case (see \cite{MonmarcheRecuitHypo} about this).

\subsection*{Application to Ornstein-Uhlenbeck processes}

\begin{cor}\label{CorOU}
 Let
\begin{eqnarray*}
Lf(x) & = & -(Bx). \na f + \text{div}\po D \na f\pf
\end{eqnarray*}
be the generator of an Ornstein-Uhlenbeck process. Suppose ker$D$ does not contain any non trivial subspace which is invariant by $B^T$, and let $M$ be the number of Lie brackets necessary to fulfill H\"ormander's condition. Suppose $\rho = \inf \{\Re(\lambda),\ \lambda\in\sigma(B)\}>0$, and let $N$ be the maximal dimension of a Jordan block of $B$ corresponding to an eigenvalue in $\{\lambda\in\sigma(B),\ \Re(\lambda) =\rho\}$.

Then the process admits a unique invariant measure $\mu$, and 
 there exist constants $c,\kappa>0$ such that for all admissible function $a$ and all $f$ with $\text{Ent}_\mu^a f < \infty$,
 \begin{eqnarray*}
 \text{\emph{Ent}}_\mu^a(P_t f) & \leq & \min\po c(1+t^{2(N-1)}) e^{-2\rho t}\ ,\ e^{-\kappa t(1-e^{-t})^{2 M }}\pf \text{\emph{Ent}}_\mu^a f
 \end{eqnarray*}
\end{cor}

\begin{proof}
According to \cite[Lemma 2.3]{AntonErb}, the condition on ker$D$ and $B$ is equivalent to
\[\sum_{k=0}^{M} B^k D \po B^T\pf^k \geq r\]
for some $r>0$. Let $C_0 = (u_1.\na,\dots,u_d.\na) := Q\na$ where the $u_i$'s are the line of a matrix $Q$ such that $Q^TQ = D$. Then $|C_0 f|^2 = (\na f)^T Q^T Q \na f = \Gamma f$. For $i\in\cco 0, M -1 \ccf$, simply let $Z_{i+1}=1$, $R_{i+1} = 0$ and  by induction $C_{i+1} = [L,C_i] = [-(Bx).\na, C_i] = Q B^{i+1} \na $. Let $Z_{M+1} = 1$, $C_{M+1} = 0$ and $R_{M + 1} = Q B^{M+1} \na$, so that
\[| R_{M+1} f|^2\hspace{5pt} = \hspace{5pt} (\na f)^T (B^{M+1})^T D B^{M+1} \na f \hspace{5pt} \leq \hspace{5pt} |QB^{M+1}|^2 |\na f|^2 \hspace{5pt} \leq \hspace{5pt} \frac{|QB^{M+1}|^2}{r} \sum_{j=0}^{M}|C_j f|^2.\]
All the conditions of Theorem \ref{TheoGeneralEllipticite} are fulfilled.

 For $\lambda \in \sigma(B)$ let $u_m \in \text{ker}(B-\lambda I)^m \setminus \text{ker} (B-\lambda I)^{m-1}$ be a generalized eigenvalue of $B$ associated to $\lambda$ of order $m\geq 1$, and for $i\in\llbracket 1,m-1\rrbracket$ let $u_{m-i} = (B-\lambda I)u_{m-i+1}$. Let $v_t = \sum_{k=1}^m \frac{t^{m-k}}{(m-k)!}u_k$, and
\[\Phi_t f \hspace{7pt}= \hspace{7pt} a''(f)|v_t . \na f|^2\hspace{7pt} =\hspace{7pt}  a''(f)(\na f)^T\frac{ \bar v_tv_t^T +  v_t \bar v_t^T }{2}  \na f.\]
Since $B v_t = \lambda v_t + v'_t$, Inequality \eqref{EqGammadeDiffusion} reads
\begin{eqnarray*}
2\Gamma_{\Phi_t} f -\partial_t \Phi_t & \geq &a''(f)(\na f)^T \co  ( \bar v_tv_t^T +  v_t \bar v_t^T)B^T - \bar v_t(v_t')^T +  v_t( \bar v_t')^T   \cf \na f \\
& = &  2 \Re(\lambda) \Phi_t f.
\end{eqnarray*}
Writing $\psi(s) =  P_s \Phi_{t-s}(P_{t-s} f)$ it means $\psi'(s) \geq  2 \Re(\lambda)\psi(s)$ and so
\begin{eqnarray*}
a''(P_t f)|v_t . \na P_t f|^2 & \leq & e^{-2\Re(\lambda) t} P_t  \po a''(f)|u_m. \na f|^2\pf.
\end{eqnarray*}
By induction, $|u_m .\na f |^2 \leq |v_t .\na f|^2 + \sum_{k=1}^{m-1} \po \frac{t^{m-k}}{(m-k)!}\pf^2 |u_k.\na f|^2$ yields
\[a''(P_tf) |u_m .\na P_t f |^2 \leq c_m (1 + t^{2(m-1)})e^{-2\Re(\lambda) t} P_t \po a''(f) \sum_{k\leq m} |u_k.\na f|^2 \pf\]
 for some constant $c_m$. Finally, considering all generalized eigenspaces, it means there exists $\tilde  c>0$ such that 
\begin{eqnarray}\label{EqOUtrouspectral}
a''(P_t)|\na P_t f |^2 & \leq & \tilde  c(1+t^{2(N-1)}) e^{-2\rho t} P_t \po a''(f)|\na f|^2 \pf.
\end{eqnarray}
An explicit  non-degenerated Gaussian invariant measure $\mu$ of the semigroup can be determined (see \cite[Lemma 3.3]{AntonErb}). Theorem \ref{TheoGeneralEllipticite} implies the semigroup admits a smooth density with respect to the Lebesgue measure. Together with the contraction \eqref{EqOUtrouspectral} for $a(f)= f^2$ and Theorem \ref{TheorKuwada}, it means for $t$ large enough $P_t$  is a contraction of the Wasserstein space $\mathcal W_2$ which is complete (\cite{BolleyWasserstein}). The fixed point Theorem ensures that $\mu$ is in fact the unique invariant law and that for all initial law $\nu$, $(\nu P_t) \overset{\mathcal W_2}\rightarrow\mu$ which mean by the Kantorovitch-Rubinstein duality that $P_t f(x) \rightarrow \mu f$ for all $x\in\R^d$ at least for any Lipschitz $f$, but then for all $f\in\A_p$ by a density argument since $\mu$ has a finite exponential moment.  Therefore we may apply Lemma \ref{LemInegalitePoincare} with $\Phi_1 f = a(f)$ and $\Phi_2 = a''(f) |\na f|^2$ to obtain that
\begin{eqnarray*}
\forall f\in\A_p\hspace{30pt}\text{Ent}_\mu^a f & \leq & K \int a''(f) |\na f|^2 \dd \mu
\end{eqnarray*} 
holds for some $K>0$. All the assumptions of Theorem \ref{TheorGeneralCoercivite} holds, and thus there exists $\kappa >0$ such that for all $f\in\A_p$
\begin{eqnarray*}
\text{Ent}_\mu^a (P_t f) & \leq &  e^{-\kappa t(1-e^{-t})^{2M }} \text{Ent}_\mu^a f
\end{eqnarray*}
On the other hand
\begin{eqnarray*}
\text{Ent}_\mu^a (P_t f) & \leq   & K \int a''(P_t f) |\na P_t f|^2 \dd \mu \\
& \leq & K e^\rho \tilde  c(1+t^{2(N-1)}) e^{-2\rho t} \int a''(P_1 f) |\na P_1 f|^2 \dd \mu \\
& \leq & c (1+t^{2(N-1)}) e^{-2\rho t}\text{Ent}_\mu^a f
\end{eqnarray*}
for some $c>0$, where the last line is due to Theorem \ref{TheoGeneralEllipticite}.
\end{proof}

\textbf{Remark:} It means the short and long time behaviours of the "distance" to equilibrium
\[d_a(P_t,\mu) := \sup\left\{ \text{Ent}_\mu^a(P_t f),\ \text{Ent}_\mu^a f = 1\right\}\]
are at least or order
\[1-d_a(P_t,\mu)  \underset{t\rightarrow 0}\sim   t^{2M+1}\hspace{45pt}d_a(P_t,\mu) \underset{t\rightarrow \infty}\sim t^{2(N-1)} e^{-2\rho t}.\]
An interpretation would be the following: at small times the law of the process instantaneously approach its equilibrium by local smoothing through diffusion, and then long-range averaging is essentially due to the drift part of the dynamics.

 The dependency in $\rho$, $N$ and $M$ is optimal: indeed, consider the kinetic Fokker-Planck process with generator
\begin{eqnarray*}
L f(x,y) & = & \po y\x - \po \frac14 x + y\pf \y + \y^2 \pf f(x,y)\\
& = & \begin{pmatrix}
x & y
\end{pmatrix}\begin{pmatrix}
0 & - \frac14 \\ 1 & -1
\end{pmatrix} \na f + \text{div} \po \begin{pmatrix}
0 & 0  \\ 0 & 1
\end{pmatrix} \na f\pf,
\end{eqnarray*}
so that $\rho = \frac12$, $N=2$ and $M=1$. 
Gadat and Miclo computed in \cite[Theorem 3]{Gadat2013} the explicit value of the (squared) operator norm of $P_t - \gamma$, namely $d_a(P_t,\mu)$ with $a(f) = f^2$, which is
\begin{eqnarray*}
\| P_t - \mu \|^2 & = & \po 1 + \frac{t^2}{2} + t\sqrt{1 + \po\frac{t}{2}\pf^2 }\pf e^{-t}\\
& \simeq & 1 - \frac{t^3}{12}\hspace{20pt}\text{as }t\rightarrow 0\\
& \simeq  & t^2 e^{-t}\hspace{30pt}\text{as }t\rightarrow\infty.
\end{eqnarray*}
In fact, as far as the long-time behaviour is concerned, this is a general fact:

\begin{prop}\label{PropOU}
In the settings of Corollary \ref{CorOU}, let $\| P_t - \mu \| = \sup\{\| P_t f - \mu f\|_{L^2(\mu)},\ \| f \|_{L^2(\mu)}=1\}$ be the $L^2$ operator norm of $P_t - \mu$. Then there exists $c>0$ such that 
\[\frac1c (1+t^{2(N-1)}) e^{-2\rho t} \hspace{7pt} \leq \hspace{7pt} \| P_t - \mu \|^2 \hspace{7pt} \leq \hspace{7pt} c (1+t^{2(N-1)}) e^{-2\rho t}\]
\end{prop}
\begin{proof}
Since $\| P_t f - \mu f \|^2_{L^2(\mu)} = \text{Ent}_\mu^a(P_t f)$ with $a(f) = f^2$, the upper bound is given by Corollary \ref{CorOU}. For the lower bound it is sufficient to exhibit a particular function $f$ such that this holds. For $u\in\R^d$, consider the linear function $f(x) = u^T x$ and let $u_t = e^{-tB^T}u$. Then  $\na (u_t^T x) = u_t$ and div$\po D \na (u_ t^T x)\pf = 0$ for all $x\in\R^d,t\geq 0$, so that $P_t f(x) = u_t^T x$ is explicit. In particular $\mu f = \underset{t\rightarrow\infty}\lim P_t f(x) = 0$, and 
$\| (P_t - \mu) f\|^2 = \int | x^Te^{-tB^T} u |^2 \mu(\dd x)$. If $u$ is a generalized eigenvector of $B^T$ or order $N$ associated to $\lambda \in \sigma(B^T)$ with $\Re(\lambda) = \rho$ and if $u_{i} = (B^T - \lambda I)^i u$ for $i\in\llbracket 0,N\rrbracket$ then $u_t = e^{-\rho t}\sum_{k=0}^{N-1} \frac{t^k}{k!} u_k$. At large time the leading term is $t^{N-1} u_{N-1}$, and so
\begin{eqnarray*}
\int | x^Te^{-tB^T} u |^2 \mu(\dd x) & \geq & \int | x^Te^{-tB^T} u |^2 \mathbb 1_{\left\{\frac{|x^Tu_{N-1}|}{|x||u_{N-1}|}\geq \frac12\right\}}\mu(\dd x)  \\
& \geq &  \gamma (1+t^{2N})e^{-2\rho t} \int |x|^2  \mathbb 1_{\left\{\frac{|x^Tu_{N-1}|}{|x||u_{N-1}|}\geq \frac12\right\}} \mu(\dd x)
\end{eqnarray*}
for some $\gamma>0$, which concludes.
\end{proof}

\textbf{Remark:} Note that the proof of Inequality \eqref{EqOUtrouspectral} does not use the positivity of $\rho$, and therefore it holds for $\rho \leq 0$. Of course in this case the process is not ergodic, but we still get a local inequality
\[a(P_t f) - P_t a(f) \leq \po \tilde  c\int_0^t (1+s^{2(N-1)}) e^{-2\rho s} \dd s\pf P_t \po a''(f)|\na f |^2\pf\]
and, in the hypoelliptic case, the Wasserstein bound
\[\mathcal{W}_2^2\po \nu_1 P_t,\nu_2 P_t\pf \leq \tilde  c (1+t^{2(N-1)}) e^{-2\rho t}\mathcal{W}_2^2\po \nu_1 ,\nu_2 \pf.\]

\section{Interacting particles on a graph}\label{SectionGraph}

\subsection{Settings}

Let $G=\{0,\dots,N\}$ be a finite non-oriented irreducible graph with edges $E\subset G^2$. We write $i\sim j$ when two vertices $i$ and $j$ are neighbours, namely when $(i,j)\in E$ (and so $(j,i)\in E$). For all $(i,j)\in E$, let $W_{i,j}=W_{j,i} \in\mathcal C^{\infty}(\R^d)$ be an even function whose all derivatives grow at most polynomially at infinity. We suppose $W_{i,j}$ is strongly convex, which means there exists $\lambda_{i,j}>0$ such that  $y^T\na^2 W_{i,j}(x)y \geq \lambda_{i,j}| y|^2$ for all $x,y\in\R^d$ where $\na^2 W_{i,j}$ is the Hessian matrix of $W_{i,j}$. We call $W_{i,j}$ the interaction potential between the sites $i$ and $j$. 

We sum up the assumptions that implicitly holds throughout Section \ref{SectionGraph}:

\begin{hyp*}\

\begin{enumerate}[(A)]
\item The graph $G$ is finite, irreducible, non-oriented.
\item For all $i\sim j$, $W_{i,j}=W_{j,i}$ is a smooth, even, $\lambda_{i,j}$-strictly convex potential with $\lambda_{i,j}>0$ whose all derivatives grow at most polynomially at infinity, and such that $|\na^2 W_{i,j}| \leqslant c_{i,j} W_{i,j}$ for some $c_{i,j}>0$. For $i\nsim j$, set $W_{i,j}=0$.
\end{enumerate}
\end{hyp*}

We are interested in $\X(t) = (X_i(t))_{i\in G} \in \R^{dN}$, a diffusion indexed by $G$. At time $t$, the particle at site $i$ undergoes an attracting force $-\na W_{i,j}\po X_i(t) - X_j(t)\pf$ from the particle at site $j$. Moreover, depending on their site, the particles undergo infinitesimal collisions from thermal motion (see \cite{EckPRB1999} for more considerations on the model). The Hamiltonian dynamics should be 
\[\forall i\in G,\hspace{15pt}\left\{ \begin{array}{rcl}
\dd X_i(t) & = & Y_i(t) \dd t\\
& & \\
m_i \dd Y_i(t) & = & - \underset{i\sim j}\sum \na W_{i,j}\po X_i(t) - X_j(t)\pf \dd t - \nu_i Y_i(t) \dd t + \overset{\infty}{\underset{j=1}\sum} \sigma_{i,j} \dd B_j(t)
\end{array}\right.\]
where $Y_i(t) \in \R^d$ is the velocity of the particle $i$, $m_i>0$ is its mass, $\nu_i>0$ a friction coefficient, $\textbf{B}(t) = \po B_j(t)\pf_{j\geq1}$ is a sequence of independent standard $1$-dimensional Brownian motions and for all $i\in G$,  $(\sigma_{i,j})_{j\geq 1}\in l^2(\R^d)$.

Following the ideas of Arnold and Erb \cite[Section 7]{AntonErb}, we could tackle this Hamiltonian process in the case where for all $(i,j)\in E$, $W_{i,j}$ is a perturbation of a quadratic potential. Nevertheless, in a first instance, we will rather consider a simpler case without any restriction on $W_{i,j}$ other than strong convexity. Taking $\nu_i = 1$ and letting the masses $m_i$ go to zero, we obtain the overdamped dynamics
\begin{eqnarray}\label{EqOverdampedParticules}
\forall i\in G,\hspace{15pt}\dd X_i(t) = - \underset{i\sim j}\sum \na W_{i,j}\po X_i(t) - X_j(t)\pf \dd t   + \overset{\infty}{\underset{j=1}\sum} \sigma_{i,j} \dd B_j(t)
\end{eqnarray}
For $x=(x_1,\dots,x_N) \in \R^{dN}$ and $i\in G$, let $\textbf{W}(x) = \sum_{k \in G}\sum_{j\sim k} W_{k,j}(x_k-x_j)$, \[b_i(x) = \sum_{i\sim j}\na W_{i,j}\po x_i - x_j\pf = \frac12\na_i \textbf{W}(x)\]
 and $b = (b_0,\dots,b_N) = \frac12 \na \textbf{W}$; let $ \sigma_j = (\sigma_{0,j},\dots,\sigma_{N,j})$ and $S = \frac12 \sum_{j\geq 1}  \sigma_j \sigma_j^T$. Then the generator associated to \eqref{EqOverdampedParticules} is
\begin{eqnarray}\label{EqGenerateurParticules}
L f(x) & = & -b(x).\na f(x) + \text{div}\po S \na f \pf.
\end{eqnarray}

Let $\overline X(t) = \frac1{N+1} \sum_{i\in G} X_i(t)$. Since $W_{i,j}=W_{j,i}$ is an even function,
\begin{eqnarray*}
\dd \overline X(t)& =& \frac{1}{N+1}\sum_{j\geq 1} \po \sum_{i\in G}\sigma_{i,j}\pf \dd B_j(t)\\
\Rightarrow\hspace{15pt} \overline X(t) & =& \overline X(0) + \frac{1}{N+1}\sum_{j\geq 1} \po \sum_{i\in G}\sigma_{i,j}\pf  B_j(t)
\end{eqnarray*}
If $\sum_{i\in G}\sigma_{i,j} = 0$ for all $j$ the process $\X$ is not hypoelliptic, since $\overline X_t$ is constant.
On the other hand if $\sum_{i\in G}\sigma_{i,j} \neq 0$ for some $j$,  $\X$ is not recurrent since $\overline X_t$ is not. There are several natural ways to force the recurrence:
\begin{enumerate}
\item We can suppose that some particles, rather than following \eqref{EqOverdampedParticules}, are fixed at the origin. In this case we can always merge these particles in order to consider only $X_0$ is fixed. We call this \emph{the fixed problem associated to \eqref{EqGenerateurParticules}}, and we call $\tX = (\widetilde X_1,\dots,\widetilde X_n)$ the associated process (more generally, if $x=(x_0,\dots,x_N)\in\R^{(N+1)d}$, we note $\tilde x = (x_1,\dots,x_N ) $). The generator $\widetilde L$ of $\tX$ is obtained from \eqref{EqGenerateurParticules} by replacing the drift $b(x)$ by $\tilde b(0,\tilde x)$ and the matrix $S = (s_{i,j})_{0\leq i,j\leq N}$ by $\widetilde S= (s_{i,j})_{1\leq i,j\leq N}$ (where we decompose $S$ as a $(N+1)\times(N+1)$ square of $d\times d$ blocks). 
\item We can add a coercive force $-\na U_i(X_i(t))$ with a strongly convex $U$ to the dynamics of some particles. In this case we can add a particle $X_{-1}(t)$ which is fixed at zero and consider that $U_i(X_i(t)) = W_{i,-1}(X_i(t) - X_{-1}(t))$, so that this case is equivalent to the previous one.
\item We can observe the cloud of particles from its center of mass, meaning that $X_i(t)$ is replaced by $X_i(t) - \overline X(t)$. We call this \emph{the centered problem associated to \eqref{EqGenerateurParticules}}, and $\hX = \X - \overline X = \po X_i - \overline X  \pf_{i=0..N}$ (more generally, if $x\in\R^d$, we note $\bar x= \frac{1}{N+1} \sum x_i$ and $\hat x = x - \bar x$). The generator $\hat L$ of the centered process is obtained from \eqref{EqGenerateurParticules} by replacing the diffusion vectors $\sigma_{j}$ by $\hat \sigma_{j} $. This process can never be hypoelliptic, and thus we may also consider $\hX  + Z $ where $Z$ is a standard Ornstein-Uhlenbeck process on $\R^d$ satisfying $\dd Z = -Z \dd t + \dd B'(t)$ where $B'$ is a standard Brownian motion on $\R^d$ which is independent from $\textbf{B}$. Note that the dynamics \eqref{EqGenerateurParticules} is invariant by translation of the center of mass: if $f$ is invariant by translation (i.e. $f(x_1+h,\dots, x_N+h) = f(x_1,\dots,x_N)\ \forall h\in\R^d$) then
\[\mathbb E\po f\po \textbf{X}(t)\pf\pf = \mathbb E\po f\po \hX (t)\pf\pf = \mathbb E\po f\po \hX + Z(t)\pf\pf.\]
\end{enumerate}



Let $\A$ be the space of $\mathcal{C}^\infty$ functions on $\R^{(N+1)d}$ whose all derivatives grow at most polynomially at infinity, $\A_0  =  \left\{f\in\A,\ \na_{x_0} f= 0\right\}$ which may be seen as a set of functions on $\R^{dN}$, and
\begin{eqnarray*}
\A_c & = & \left\{f\in \A,\ f(x+h) = f(x)\ \forall x\in \R^{(N+1)d},\ h\in \R^d\right\} \hspace{7pt}=\hspace{7pt}\left\{f\in \A,\ \sum_0^{N} \na_{x_i} f = 0\right\}
\end{eqnarray*}

\begin{prop}
For any initial law, the process $\X$ (resp. $\tX$) is well-defined for all times. The associated semi-group $P_t f(x) = \mathbb E\po f(\X (t))\ |\ \X(0) = x \pf$ (resp. $\widetilde P_t$) is Feller and fix $A_c$ (resp. $A_0$). 
\end{prop}

\begin{proof}
The convexity and growth conditions on $\textbf{W}$ ensures that $\phi(x) = 1 + \textbf{W}(x) +(\bar x)^2$, which is such that $|x|^2/\phi(x)$ is bounded, satisfies
\[L\phi \ \leqslant \ -\frac12 |\na \textbf{W}|^2 + A \phi  \ \leqslant \  A \phi\]
for some $A>0$, which implies equation \eqref{EqOverdampedParticules} admits a unique strong solution $\X$ (see \cite[Theorem 3.1]{Durrett} ; the same holds for the fixed problem with $\phi = 1 +\widetilde{\textbf{W}}$) and the associated semi-group is Feller (see \cite[\emph{(22.5) p. 164}]{RogersWilliams2}). If $f$ is $\mathcal C^\infty$, so is $P_t f$ for all $t\geq 0$ (see \cite[Theorem VII.5]{Brezis}). The existence of the Lyapunov function $\phi$ implies the moments of $\X$ are finite for all time. Indeed, for $m\in\mathbb N$ we see that, for some $A'$, $A''>0$,
\begin{eqnarray*}
L \phi^m  & = & m \phi^{m-1} L\phi + m(m-1) \phi^{m-2} \Gamma \phi\\
& \leqslant & m \phi^{m-1} \po -\frac12 |\na \textbf{W}|^2 + A \phi \pf + A' \phi^{m-2} \po |\na \textbf{W}|^2+ (\bar x)^2\pf\\
& \leqslant & A'' \phi^m,
\end{eqnarray*}
where we used that $\phi$ goes to infinity at infinity. As a classical consequence (see e.g. the proof of \cite[Lemma 2.1]{Talay}), $P_t \phi^m \leqslant e^{A'' t} \phi^m$, which means if $f$ grows at most polynomially at infinity, so is $P_t f$ for $t\geq0$. If $n$ is a multi-index, differentiating with respect to the space variables the Kolmogorov equation $\partial_t P_t f = L P_t f$ yields, from \cite[Theorem VII.10]{Brezis},
\begin{eqnarray*}
\partial_x^n P_t f & =& P_t \partial_x^n f + \int_0^t P_{t-s} [\partial^n,L] P_s f \dd s.
\end{eqnarray*}
For some $C,k>0$ depending on the derivatives of $W$, $|[\partial^n,L] g|(x) \leq C(1+x^k)\sum_{|j|<|n|}|\partial^j g|(x)$. Therefore by induction on $|n|$, if $\partial_x^j  f$ grow at most polynomially at infinity for all $|j|<|n|$, so does $\partial_x^n P_t f$. The case of $\widetilde P_t$ is similar.
\end{proof}


\subsection{Main results}

We note
\begin{eqnarray*}
L_G h(i) & = & \sum_{j\sim i} \lambda_{ij}\po h(j) - h(i) \pf
\end{eqnarray*}
the discrete Laplacian on $G$ (also seen as an $(N+1)\times (N+1)$ matrix). This is the generator of an ergodic Markov chain on $G$ (recall the $\lambda_{ij}$'s are positive and $G$ is irreducible). Since $L_G$ is symmetric the invariant measure is the uniform law $\nu$ (also seen as the column vector $\frac1{N+1}(1,\dots,1)$) on $G$. Let
\begin{eqnarray*}
\rho & = & \underset{h\in\R^{N+1}}{\inf} \frac{- h^T L_G h}{|h-\nu h|^2} 
\end{eqnarray*} 
be the spectral gap of $L_G$ and 
\begin{eqnarray*}
\rho_{D} & = & \underset{h\in\R^{N+1},\ h_0 = 0}{\inf} \frac{- h^T L_G h}{|h|^2} 
\end{eqnarray*}
be its Dirichlet eigenvalue.

\begin{prop}\label{PropGraph}
For the fixed problem associated to \eqref{EqGenerateurParticules}, for all admissible function $a$ and for all $t\geq 0,f\in\A_0$,
\begin{eqnarray*}
a''(\widetilde P_t f)|\na \widetilde P_t f |^2  & \leq & e^{-2\rho_D t} \widetilde P_t\po a''(f)|\na  f |^2\pf. 
\end{eqnarray*}
For the centered problem associated to \eqref{EqGenerateurParticules}, for all $t\geq 0,f\in\A_c$,
\begin{eqnarray*}
a''(P_t f) |\na P_t f |^2  & \leq & e^{-2\rho t}  P_t\po a''(f)|\na  f |^2\pf 
\end{eqnarray*}
\end{prop}

\begin{proof}
For $\Phi(f) = a''(f) |\na f|^2$, we computed in \eqref{EqGammadeDiffusion} that
\[\Gamma_{\hat L,\Phi} f \geq a''(f) (\na f)^T J_b \na f\hspace{20pt}\text{and}\hspace{20pt}\Gamma_{\tilde L,\Phi} f \geq a''(f) (\na f)^T J_{\tilde b} \na f\]
where $J_b$ and $J_{\tilde b}$ are the Jacobian matrices of the drifts $b$ and $\tilde b$. We compute
\begin{eqnarray*}
(\na f)^T J_b \na f & =& \sum_{i\in G}\sum_{j\sim i} (\na_{x_i} f)^T \po \na^2 W_{i,j}\pf (\na_{x_i} -\na_{x_j}) f\\
& =& \sum_{(i,j)\in E} \po (\na_{x_i} - \na_{x_j} ) f\pf^T \po \na^2 W_{i,j}\pf (\na_{x_i} -\na_{x_j}) f\\
& \geq & \sum_{(i,j)\in E} \lambda_{i,j} |(\na_{x_i} - \na_{x_j} ) f|^2\\
& = & -(\na f)^T L_G \na f,
\end{eqnarray*}
and similarly $(\na f)^T J_{\tilde b} \na f \geq - (\na f)^T  L_G \na f$ for $f\in \A_0$. When $f\in\A_c$, $\nu \na f =  0$, and thus $\Gamma_{\hat L,\Phi} f \geq \rho \Phi(f)$. When $f\in\A_0$,  $\Gamma_{\widetilde L,\Phi} f \geq \rho_D \Phi(f)$. In both case (and since $\A_c$ and $\A_0$ are respectively fixed by $P_t$ and $\widetilde P_t$), Lemma \ref{LemGammaCurvature} concludes.
\end{proof}

Let
\begin{eqnarray*}
\Theta_N & = & \left\{x\in\R^{(N+1)d},\ \bar x = 0\right\}
\end{eqnarray*}
so that $\hX \in \Theta_N$ for all $t\geq 0$, and thus the associated semi-group $\widehat{P_t}$ acts on  $\mathcal P\po \Theta_N\pf$.

\begin{cor}\label{CorGraph}
Suppose in addition to Assumptions $(A)$ and $(B)$ that the processes $\X$ and $\tX$ are hypoelliptic. Then for any initial laws $\nu_1,\nu_2 \in \mathcal P\po \Theta_N\pf$,
\begin{eqnarray*}
\mathcal W_2\po \nu_1 \widehat{P_t}, \nu_2 \widehat{P_t} \pf & \leq & e^{-\rho t}\mathcal W_2\po \nu_1 , \nu_2  \pf  
\end{eqnarray*}
and for any initial laws $\nu_1,\nu_2 \in \mathcal P\po \R^{dN}\pf$,
\begin{eqnarray*}
\mathcal W_2\po \nu_1 \widetilde P_t, \nu_2 \widetilde  P_t \pf & \leq & e^{-\rho_D t}\mathcal W_2\po \nu_1 , \nu_2  \pf.
\end{eqnarray*}
Moreover both the fixed and the centered problems admits a unique invariant law, respectively denoted by $\mu_0$ and $\mu_c$, which satisfy the log-Sobolev inequalities
\begin{eqnarray*}
\forall f\in\A_c \ s.t.\ \int f \dd \mu_c = 1,\hspace{30pt}\int f \ln f \dd \mu_c & \leq & \frac{|S|}{2\rho}\int \frac{|\na f|^2}f \dd \mu_c\\
\forall f\in\A_0 \ s.t.\ \int f \dd \mu_0 = 1,\hspace{30pt}\int f \ln f \dd \mu_0 & \leq & \frac{|\tilde S|}{2\rho_D}\int \frac{|\na f|^2}f \dd \mu_0
\end{eqnarray*}
\end{cor}

\begin{proof}
Both cases are similar and thus we only consider the first one. The Wasserstein contractions is implied by Theorem \ref{TheorKuwada} and Proposition \ref{PropGraph} for $a(f) = f^2$ (recall that for $f\in\A_c$, $P_t f = \widehat{P_t} f$). When $E$ is a Polish space, so is $\mathcal P_2(E) = \{\nu \in \mathcal P(E) \text{ having a finite $2^{nd}$ moment}\}$ endowed with the $\mathcal W_2$ metric (\cite{BolleyWasserstein}). Therefore by the fixed point Theorem, $\nu = \nu \widehat{P_t}$ admits a unique solution in $\mathcal P_2(\Theta_N)$ which may a priori depend on $t$, but then $\nu \widehat{P_s} = \nu \widehat{P_t} \widehat{P_s} = \nu \widehat{P_s} \widehat{P_t}$, and by uniqueness of the solution $\nu \widehat{P_s} = \nu$ for all $s\geq 0$. The existence of the Lyapunov function $\textbf W$ implies an invariant measure necessarily have a finite second moment.

The log-Sobolev inequality is implied by Proposition \ref{PropGraph} for $a(x) = x \ln x$ together with Lemma \ref{LemInegalitePoincare} for $\gamma(t) = e^{-2\rho t}$, $\phi_1(f) = f \ln f$ and $\phi_2(f) = \frac{|\na f|^2}{f}$ so that 
\[\Gamma_{\Phi_1} f  =  \frac{(\na f)^T S \na f}{2f} \leq \frac12|S| \Phi_2(f). \]
\end{proof}

\textbf{Remark:} since the drift of the diffusion is of the form $-\na \textbf W$ with a convex $\textbf{W}$, there would be no difficulties to apply here the Meyn-Tweedie-Lyapunov techniques for proving convergence in total variation distance (cf. \cite{CattiauxGuillinPAZ,Cattiaux2008}) under  controllability or strong  hypoellipticity conditions. However the "explicit" speed of convergence obtained this way depend on estimates of the probability transition of the process on a compact set, which are much less trackable than the constant involved in Corollary \ref{CorGraph}. The use of an explicit mirror coupling (see \cite{CattiauxGuillin2014}) may give satisfactory estimates.

\subsection{Chain of particles}

Suppose $i\sim j$ iff $|i-j|\leq 1$, and $W_{i,j} = W$  does not depend on the edge $(i,j)$ (and so does $\lambda_{i,j}=\lambda$).
 We call this dynamics \emph{the chain of $N+1$ particles}.

\begin{prop}\label{PropChain1}
For the chain of $N+1$ particles, for $N\geq 3$,
\[\rho \geq \frac{\lambda}{(N+1)^2}\]
\end{prop}
\begin{proof}
Let $h = \min\left\{\frac{|\partial A|}{|A|},\ |A|<\frac{|G|}2\right\}$ be the Cheeger constant of $G$, where $\partial A \subset E$ is the set of (non-oriented) edges $(i,j)$ with $i\in A$ and $j\notin A$, and $|B|$ denotes the cardinal of a set $B$. Since the maximal degree of $G$ is 2, from \cite[Theorem 4.2]{Mohar}, for $N\geq 3$,
\begin{eqnarray*}
\rho & \geq & \lambda \frac{h^2}4.
\end{eqnarray*}
It is clear the minimum of $\frac{|\partial A|}{|A|}$ is attained with $A=\{0,\dots,\lfloor (N+1)/2 \rfloor \}$, so that $h \geq \frac{2}{N+1}$.
\end{proof}

\begin{prop}\label{PropChain2}
For the chain of $N+1$ particles,
\[\rho_D \geq \lambda \po 1 - \cos\po \frac{\pi}{2N}\pf \pf\]
\end{prop}
\begin{proof}
 Consider $(Y_t)_{t\geq 0}$ the continuous time nearest neighbour random walk on $\{0,\dots,N\}$ starting at $N$ and absorbed at 0, whose generator is $\frac1{\lambda}\widetilde L_G$. The dynamics is the following: from the time $t_0$, $Y$ waits a time $\frac12 E$ with standard exponential law, and then jump to $Y_{t_0+E}$ equal to either $Y_{t_0}+1$ or $Y_{t_0}-1$ with equal probability $\frac12$, unless $Y_{t_0} = 0$ in which case it does not jump, or unless $Y_{t_0} = N$ in which case it jumps to $N-1$ or stays at $N$ with equal probability $\frac12$ (or in other words, it waits a time $E$ to bounce back to $N-1$). According to the work \cite{DiaconisMicloQuasiStat} of Diaconis and Miclo, if $T$ is the time of absorption of $Y$, then
\[\underset{t\rightarrow\infty} \lim \frac1t \mathbb P\po T > t \pf  = - \rho_D.\]
The process $Y$ is very close to the process $\widetilde Y$ of \cite[Example 16]{DiaconisMicloQuasiStat}, the only difference (up to renumbering the states) being that when $\widetilde Y_{t_0}=N$, the process $\widetilde Y$ only waits a time $\frac12 E$ to bounce back to $N-1$. As a consequence, $\widetilde Y$ goes faster to $N$ that $Y$, but $(Y_{2t})_{t\geq 0}$ goes faster to $N$ than $\widetilde Y_t$ and 
\begin{eqnarray*}
\underset{t\rightarrow\infty} \lim \frac1t \mathbb P\po T > t \pf & \leq & \frac12\  \underset{t\rightarrow\infty} \lim \frac1t \mathbb P\po \widetilde T > t \pf
\end{eqnarray*}
where $\widetilde T$ is the absorbing time of $\widetilde Y$. Now explicit spectral computations (cf. \cite{DiaconisMicloQuasiStat}) yields
\begin{eqnarray*}\underset{t\rightarrow\infty} \lim \frac1t \mathbb P\po \widetilde T > t \pf & = & -2\po 1 - \cos\po \frac{\pi}{2N}\pf\pf.
\end{eqnarray*}

\end{proof}
\textbf{Remark:} As an alternative proof, we could have estimated the Laplace transform of $T$. Indeed Chernoff's inequality yields $\mathbb P\po T > t \pf \leq \mathbb E \po e^{\theta T} \pf e^{-\theta t}$. Denoting by $T_k$ the absorbing time starting from $k$ and $b_k = \mathbb E \po e^{\theta T_k}\pf\po \mathbb E \po e^{\theta T_N}\pf\pf^{-1}$, a recurrence relation satisfied by the $b_k$'s and the fact $T_0= 0$ allow to prove that for any $\eta<\frac{\pi}{2N}$, $\mathbb E \po e^{(1-\cos(\eta)){\lambda}T_k}\pf = \frac{1}{\cos(k\eta)}$. This means $\rho_D \geq \lambda\po 1 - \cos\po \frac{\pi}{2N}\pf\pf$. 

\bigskip

Now consider a chain of particles for which all the diffusion blocks $s_{i,j}$ are zero except $s_{N,N}$ and $s_{0,0}$ which are homogeneous dilations. In other words, we are interested in the processes $\X$ and $\tX$ defined by
\begin{eqnarray}\label{Eq25}
& & \left.\begin{array}{rcl}
\dd Z_i & = & - \po \na W(Z_i - Z_{i-1}) + \na W(Z_i - Z_{i+1})\pf \dd t, \hspace{20pt}\forall i\in\cco 1,N-1 \ccf,\ Z_i = X_i\text{ or } \widetilde X_i\\
& & \\
\dd Z_N & = & -  \na W(Z_N - Z_{N-1}) \dd t + \sigma_N \dd B_1(t), \hspace{25pt}\ Z_N = X_N\text{ or } \widetilde X_N\\
& & \\
\dd X_0 & = & -  \na W(X_0 - X_i) \dd t + \sigma_0 \dd B_2(t)\\
& & \\
\widetilde X_0 & = & 0.
\end{array}  \right\}
\end{eqnarray} 
where $B_1$ and $B_2$ are independent $d$-dimensional Brownian motion and $\sigma_0,\sigma_N\in\R$. 

\begin{lem}
The processes $\X$  and $\tX$ defined by \eqref{Eq25} are hypoelliptic. Moreover, if $\na^2 W$ is bounded, then the assumptions of Theorem \ref{TheoGeneralEllipticite} are satisfied.
\end{lem}
\begin{proof}
Again both cases are similar and we only treat the centered problem. Let $V_0 = \sigma_N^2 \na_N$ and for $i\in\cco 1,N\ccf$ let $V_i = \sigma_N^2\na^2 W(x_{N-1} - x_{N})\dots \na^2 W(x_{N-1-i} - x_{N-i}) \na_{N-i}$. Then $[V_i,\na \textbf{W}.\na] = V_{i+1} + r_i$ where $r_i(x)\in \text{span}\{V_j(x),\ j \leq i\}$ for all $x\in\R^{d(N+1)}$ and the parabolic H\"ormander's condition is fulfilled for $\X$.

If moreover $\na^2 W$ is bounded, from 
\[\co - \na \textbf W.\na ,\na_j \cf = \na^2 W(x_{j}-x_{j-1}) \po\na_{j-1} - \na_j\pf +  \na^2 W(x_{j+1}-x_{j}) \po\na_{j+1} - \na_j\pf,\]
the assumptions of Theorem \ref{TheoGeneralEllipticite} hold with $C_i = \na_{N-i}$ and $m=2\Lambda  =2\| \na^2 W\|_{\infty}$.
\end{proof}

As a conclusion,  Corollary \ref{CorGraph} may be applied to $\X$  and $\tX$  with $|S| = max\po \sigma_N^2,\sigma_0^2\pf$, $|\widetilde S|=\sigma_N^2$, and with $\rho$ and $\rho_D$ bounded in Propositions \ref{PropChain1} and \ref{PropChain2}. If moreover $\na^2 W$ is bounded, both Theorems \ref{TheoGeneralEllipticite} and \ref{TheorGeneralCoercivite} holds and we get for $\X - \bar X$ and $\tX$ convergences of the form
\begin{eqnarray*}
\text{Ent} P_t f  & \leq &  c N^{c' N^2} e^{- \frac{c''}{N^2} t} \text{Ent} f
\end{eqnarray*}
where $c,c'$ and $c''$ do not depend on $N$.

\section*{Acknowledgments}

The author would like to thank Fran{\c{c}}ois Bolley for pointing out a mistake in a early version of the paper. This work has been supported by EFI project ANR-17-CE40-0030 of the French National Research Agency (ANR).

\bibliographystyle{plain}
\bibliography{biblio}

\end{document}